\documentclass[12pt,a4paper,twoside]{article}
\usepackage{amsthm, amsfonts,amsmath, amssymb, color, bm}

\newtheorem{theorem}{Theorem}
\newtheorem{definition}{Definition}
\newtheorem{example}{Example}
\newtheorem{lemma}{Lemma}
\newtheorem{proposition}{Proposition}
\newtheorem{remark}{Remark}

\def\eq{\hspace*{-.1mm}&=&\hspace*{-.1mm}}

\def\Div{\operatorname{div}}

\newcommand\const{\operatorname{const}}
\def\R{\mathbb{R}}
\def\N{\mathbb{N}}

\title{Bochner's technique in Einstein's non-symmetric geometry}
\author{Vladimir Rovenski
\footnote{Department of Mathematics, University of Haifa, Mount Carmel, 3498838 Haifa, Israel
\newline e-mail: {\tt vrovenski@univ.haifa.ac.il}
}}

\date{}

\begin{document}

\maketitle

\abstract{
A. Einstein considered a manifold with a non-symmetric (0,2)-tensor $G=g+F$, where $g$ is a Riemannian metric and $F\ne0$, and a connection $\nabla$ with torsion $T$ such that $(\nabla_X G)(Y,Z)=-G(T(X,Y),Z)$.
Guided by the almost Lie algebroid construction on a vector bundle, we define the basic concepts of Bochner's technique for Einstein's non-symmetric geomet\-ry,
give a clear example of the Einstein's connection $\nabla$,
prove Weitzen\-b\"{o}ck type decomposition formula and obtain vanishing results about the null space of the Bochner and Hodge type~Laplacians.


}


\section{Introduction}

Since 1923,
A. Einstein worked on various versions of the Unified Field Theory (Non-symmetric Gravitational Theory -- NGT), see~\cite{Ein}. This theory was intended to unite the gravitation theory, to which General Relativity (GR) is related, and the theory of electromagnetism.
Introducing various versions of his NGT, A. Einstein used a complex basic tensor, with a symmetric real part and a skew-symmetric imaginary part. Beginning in 1950, A. Einstein used a real non-symmetric basic (0,2)-tensor
$G$
whose symmetric part $g$
is associated with gravity, and the skew-symmetric part $F$
is associated with electromagnetism.
The same is true for the symmetric part of the connection and torsion tensor, respectively.
In his attempt to construct an unified field theory,
(Einstein's Gravitational Theory, briefly NGT), A.~Einstein \cite{Ein} considered a manifold $(M,G=g+F)$, 
where $g$ is a Riemannian metric and $F\ne0$, with a linear connection $\nabla$
with torsion $T$ such that
$(\nabla_X G)(Y,Z)=-G(T(X,Y),Z)$
 (called an \textit{Einstein's connection}),
or, equivalently, see \cite{IZ1,rst-63}:
\begin{equation}\label{metein}
 XG(Y,Z)-G(\nabla_YX,Z)-G(Y,\nabla_XZ)-G([X,Y],Z)=0
 \quad (X,Y,Z\in\mathfrak{X}_M) .
\end{equation}
A manifold $(M,G=g+F)$ equipped with an Einstein's connection $\nabla$ different from the Levi-Civita connection $\nabla^g$
is called an NGT space, or \textit{Einstein's non-symmetric geometry}.
The idea of a non-symmetric basic tensor was later revisited by J.W. Moffat \cite{Moffat-95}, allowing for richer gravi\-tational dynamics.
Results by T. Jansen and T. Prokopec \cite{JP-2007}
suggest potential improvements and challenges within NGT.
M.I.~Wanas et al. \cite{Wanas-24} studied
path equations of a particle in Einstein's non-symmetric geo\-metry.
Recent approaches to modified gravity often rely on differen\-tial geometry, incorporating torsion and non-metricity as natural extensions of GR framework.
S.~Ivanov and M.~Zlatanovi\'c~\cite{IZ1}
investigated NGT using connections with totally skew-symmetric torsion and gave examples with the skew-symmetric part $F$ of $G$.
V.~Rovenski and M.~Zlatanovi\'c~\cite{rst-63} were the first to apply weak $f$-structures, in particular, weak almost contact structures, to NGT of totally skew-symmetric torsion with
$F(X,Y) = g(X,fY)$.
Weak metric structures
(developed in recent works by V.~Rovenski, see surveys \cite{rov-survey24,rst-139})
generalize the classical structures and are well suited for a 
study of $G=g+F$ in NGT with an arbitrary skew-symmetric tensor~$F$.


The Bochner technique in differential geometry and mathematical physics is a powerful analytic method used to prove vanishing theorems
by relating topology of a complete Riemannian manifold to its curvature.
The Bochner technique works for skew-symmetric tensors lying in the kernel of the Hodge Laplacian
$\Delta_H=d\,\delta+\delta\,d$:
using maximum principles, one proves that such tensors are parallel, e.g.,~\cite{Peter,rov-125}.
Here $d$ is the exterior differential operator, and $\delta$ is its adjoint operator for the $L^2$ inner product.
The
elliptic differential operator $\Delta_H$
decomposes into two~terms,
\begin{equation}\label{E-Wei0}
 \Delta_H =\nabla^*\nabla+\Re,
\end{equation}
one is the \textit{Bochner Laplacian} $\nabla^*\nabla$, and the second term
(depends linearly on the Riemannian curvature tensor),
is called the \textit{Weitzenb\"{o}ck curvature operator} on $(0,k)$-tensors $S$, e.g.,~\cite{Peter}, where
\begin{eqnarray}\label{E-Ric}
 && \Re\,(S)(X_{1},\ldots,X_{k}) = \sum\nolimits_{a=1}^{k}\sum\nolimits_{\,i=1}^{n}
 (R_{\,e_{i},X_{a}}\,S)(\underbrace{X_{1},\ldots, e_{i}}_{a},\ldots, X_{k}).
\end{eqnarray}
Here
$\{e_{i}\}$ is a local orthonormal frame on $(M,G=g+F)$ and $\nabla^*$ is the $L^2$-adjoint of the Levi-Civita connection $\nabla$.
Note that $\Re$ reduces to ${\rm Ric}$
 when evaluated on (0,1)-tensors, i.e., $k=1$.
According to the well-known formula $(R_{Z,Y}\,S)(X_1,\ldots,X_k) = -\sum\nolimits_{\,i}S(X_1,\ldots R_{Z,Y}X_i,\ldots,X_k)$
for the action of the curvature tensor $R$ on $(0,k)$-tensors, for $k\ge2$ the formula from \eqref{E-Ric} has the form
\begin{align*}
 \Re(S)(X_{1},\ldots,X_{k})
 =& -2\sum\nolimits_{\,i,j,a;b<a}R(e_i,X_a,e_j,X_b)\cdot S(\underbrace{X_1,\ldots,e_j}_{b},\underbrace{\ldots, e_i}_{a-b},\ldots,X_k)\\
 & +\,\sum\nolimits_{\,i,a}{\rm Ric}(e_i,X_a)\cdot S(\underbrace{X_1,\ldots, e_i}_{a},\ldots, X_k).
\end{align*}
%
The formula \eqref{E-Wei0} allows us to extend the Hodge Laplacian to arbitrary tensors.


In \cite{rp-2,rp-3}, we expanded the Bochner's technique to a Riemannian manifold endowed with a singular or regular distribution and Levi-Civita or a statistical type connection,
and generalized some results
in~\cite{opozda1}.
 Singularities play a crucial role in mathematics
and its applications,
e.g., in theories of nonlinear control systems and GR.
A \textit{singular distribution} ${\cal D}$ on a manifold $M$ assigns to each point $x\in M$ a linear subspace ${\cal D}_x$
of the tangent space $T_xM$ in such a way that, for any $v\in{\cal D}_x$,
there exists a smooth vector field $V$ defined in a neighborhood $U$ of $x$
and such that $V(x) = v$ and $V(y)\in{\cal D}_y$ for all $y\in U$.
A~priori, the dimension $\dim{\cal D}_x$ depends on $x\in M$. If $\dim{\cal D}_x=\const$, then ${\cal D}$ is a regular distribution.
The distribution ${\cal D}={f}(TM)$ is singular in the sense of the definition given above.
Indeed, for any vector $v\in{f}(T_xM)$ there exists a vector $\tilde v\in T_xM$ such that ${f}(\tilde v)=v$.
Take a smooth vector field $\tilde V$ in a neighborhood $U$ of $x$ such that $\tilde V(x)=\tilde v$.
Then $V={f}(\tilde V)$ is a smooth vector field in $U$ and such that $V(x)=v$ and $V(y)\in{f}(T_y M)$ for all~$y\in U$.


In the paper, we consider Einstein's connections, see \eqref{metein}, satisfying the additional condition $\nabla g=0$,
hence the properties (ii) and (iii) of Lemma~\ref{L-nabla-g} below hold.
We develop the Bochner technique for NGT spaces:
the $f$-connection, $f$-divergence, $f$-curvature tensor,
$f$-Laplacians and Weitzenb\"{o}ck $f$-curvature operator (where $f$, see \eqref{E-def-f}, is conjugate to $F$).
Our results hold both  in the case where $f$ is non-degenerate and  in the case where $\ker f$ defines a singular
distribution (i.e, $\dim\ker f>0$ at some points), which applies to the weak metric $f$-structure, see~\cite{rov-survey24,rst-139}.


The work has the Introduction and two sections, the References include 16 items.
In~Section~\ref{sec:prel}, we recall the basics of Einstein's non-symmetric geometry and the almost Lie algebroid structure.
Section \ref{sec:main} consists of five subsections.
In Section~\ref{sec:nabla}, for a mani\-fold $(M,G=g+F)$ with an Einstein's connection $\nabla$,
we~define the derivative $\nabla^{f}$, the $f$-bracket and give a clear example of the Einstein's connection $\nabla$. 
In Section~\ref{sec:div}, we define the skew-symmetric $f$-divergence and their $L^2$ adjoint operators on tensors.
In Section~\ref{sec:laplace}, we introduce Bochner and Hodge $f$-Laplacians on tensors and differential forms.
In~Sections~\ref{sec:R} 
and \ref{sec:Ric}, we define 
the $f$-curvature operator $R^f$ of $\nabla^{f}$, the Weitzenb\"{o}ck type curvature operator on tensors,
prove Weitzen\-b\"{o}ck type decomposition formula
and obtain vanishing~results.

\section{Preliminaries}
\label{sec:prel}

Here, we discuss
Einstein's 
non-symmetric geometry
and the almost Lie algebroid~structure.


\subsection{Einstein's non-symmetric geometry}

The basic (0,2)-tensor $G$ in a ``generalized" Riemannian manifold $(M,G)$ is non-symmetric.
It~decomposes in two parts, $G=g+F$, the symmetric part $g>0$ (called Riemannian metric) and the
skew-symmetric part $F\ne0$ (called fundamental 2-form), where
\begin{equation}\label{metric}
g(X,Y)=\frac12\big[G(X,Y)+G(Y,X)\big], \quad F(X,Y)=\frac12\big[G(X,Y)-G(Y,X)\big].
 \end{equation}
Therefore, we obtain a
(1,1)-tensor ${f}\ne0$ (of not necessarily constant rank) determined by
\begin{equation}\label{E-def-f}
 g(X, fY) = F(X,Y)\quad (X, Y \in \mathfrak{X}_M).
\end{equation}
Since $F$ is skew-symmetric, the tensor ${f}$ is also skew-symmetric:
 $g({f}X, Y) = -g(X,{f}Y)$.
We~consider linear connections $\nabla$ on
$M$
with a torsion (1,2)-tensor
\[
 T(X,Y)=\nabla_XY-\nabla_YX-[X,Y].
\]
Denote the torsion (0,3)-tensor with respect to $g$ by the same letter,
 $T(X,Y,Z):=g(T(X,Y),Z)$.


 The covariant derivative an $(p,k)$-tensor $S$ for $p=0,1$, as an $(p,k+1)$-tensor $\nabla S$, see \cite{Peter}:
\begin{equation}\label{E-nablaP-S0}
 (\nabla S)(Y,X_{1},\ldots,X_{k}) =
 \nabla_{Y}(S(X_{1},\ldots,X_{k}))
 -\sum\nolimits_{\,i=1}^{k}S(X_{1},\ldots,
 \nabla_{Y}X_{i},\ldots,X_{k}) .
\end{equation}
Using \eqref{metric}, \eqref{E-def-f}
and \eqref{E-nablaP-S0},
we present the condition \eqref{metein} in the following form, see \cite{IZ1}:
\begin{align*}
 (\nabla\,g)(X, Y,Z)
 &=-\frac12\,\big[T(X,Y,Z) +T(X,Z,Y) +T(X,Y,{f}Z)+T(X,Z,{f}Y)\big],\\
(\nabla\,F)(Z, X,Y)&= \frac12\,\big[ T(X,Z,Y)-T(X,Y,Z) -T(X,Y,{f}Z)+T(X,Z,{f}Y)\big].
\end{align*}
The above equations yield, see  \cite[Eq.~(3.4)]{IZ1},
\begin{align}\label{ein7}
 (\nabla\,g)(Z, X, Y) = (\nabla\,F)(X,Y,Z) +(\nabla\,F)(Y,X,Z).
\end{align}

\begin{lemma}\label{L-nabla-g}
For an Einstein's connection $\nabla$, the following conditions are equivalent:

$(i)$ $\nabla g=0$,

$(ii)$ $K(X,Y,Z)=-K(X,Z,Y)$,

$(iii)$ $(\nabla_X F)(Y,Z)=-(\nabla_Y F)(X,Z)$,
or $(\nabla_X f)Y=-(\nabla_Y f)X$.
\end{lemma}

\begin{proof}
 We calculate using \eqref{E-nablaP-S0} and \eqref{eq:contorsion-def},
\begin{align*}
 (\nabla_X\,g)(Y,Z) &
 = X(g(Y,Z)) - g(\nabla^g_X Y+ K_XY, Z) - g(Y, \nabla^g_X Z+ K_XZ) \\
 & = (\nabla^g_X\,g)(Y,Z) - g(K_XY, Z) - g(K_XZ, Y).
\end{align*}
By this and $\nabla^g\,g=0$, the conditions (i) and (ii)
are equivalent.
By the equality \eqref{ein7},
the conditions (i) and (iii) are equivalent.
\end{proof}

The \emph{difference} tensor \( K \)
of a linear connection $\nabla$
and
$\nabla^g$ is given~by
\begin{equation}\label{eq:contorsion-def}
 K_X Y := \nabla_X Y - \nabla^g_X Y.
\end{equation}
It is called the \textit{contorsion tensor} when $\nabla$ preserves the metric tensor, $\nabla g=0$, in Riemann-Cartan geometry.
Denote the torsion (0,3)-tensor with respect to $g$ by the same letter, $K(X,Y,Z):=g(K_X Y, Z)$.
In the case $\nabla g=0$, the tensors $K$ and $T$ are expressed linearly through each~other:
\begin{align}\label{E-tordfnew2}
  T(X,Y) &= K_X Y - K_Y X , \\
 \label{E-tordfnew3}
  2\,K(Y,Z,X) &= T(X,Y,Z) + T(Y,Z,X) + T(X,Z,Y).
\end{align}

The Einstein's connection $\nabla$
is represented in \cite{IZ1} as
\begin{align}\label{genconein}
 g(\nabla_XY,Z) &=g(\nabla^g_XY,Z) +\frac12\big[T(X,Y,Z) +T(X,Z,{f}Y) +T(Y,Z,{f}X)\big].
\end{align}
By \eqref{genconein}, the tensor $K$ and the torsion $T$
of an Einstein's connection $\nabla$ are related as
\begin{equation}\label{E-tordfnew}
 2\,K(X,Y,Z) = T(X,Y,Z)+T(X,Z,{f}Y) +T(Y,Z,{f}X) .
\end{equation}
Comparing \eqref{E-tordfnew3} and \eqref{E-tordfnew} yields the
condition for the torsion of an Einstein's connection
\begin{align*}
 T(Z,X,Y)
 + T(Z,Y,X) =
 T(X,Z,{f}Y) +T(Y,Z,{f}X).
\end{align*}


\begin{remark}
\rm
Linear connections with \textit{totally skew-symmetric torsion},
 $T(X,Y,Z)=-T(X,Z,Y)=-T(Y,X,Z)$,
have become popular due to the relations with mathematical physics. In this case, the number of preserved supersymmetries is equal to the number of parallel spinors with respect to such a connection.
S.~Ivanov S. and M.Lj.~Zlatanovi\'c, see \cite{IZ1}, presented conditions for the existence and uniqueness of the Einstein's connection with totally skew-symmetric torsion on a manifold $(M,G=g+F)$ and gave its explicit expression.
In this case, the torsion is determined by~$dF$:
 $T(X,Y,Z)=-\frac13dF(X,Y,Z)$,
and the difference tensor is given by the following formula:
\begin{equation}\label{newnbl-K}
K_XY=\frac12\big[T({f}X,Y)-T(X,{f}Y)+T(X,Y)\big].
\end{equation}
\end{remark}

\begin{example}\label{Ex-3.9a}
\rm
A \textit{weak almost Hermitian manifold} $(M, f, Q, g)$ is a Riemanni\-an manifold $(M, g)$ of dimension $2m \ge 4$ endowed with non-singular endomorphisms: $f$ (skew-symmetric) and $Q$ (self-adjoint), and the 2-form $F$
such that the following conditions are valid, see~\cite{rst-63}:
\begin{equation}\label{WAH}
f^2 = -Q, \quad g(fX, fY) = g(QX, Y),
\quad F(X,Y)=g(X,fY).
\end{equation}
Let $(M, f, Q, g)$ be a \textit{weak nearly K\"{a}hler manifold},
i.e., \eqref{WAH} and $(\nabla^g_X f)X=0$ are true,
considered as a generalized Riemannian manifold $(M, G = g + F)$.
Suppose that $\nabla$ is an Einstein's connection
on $(M, f, Q, g)$ with totally skew-symmetric torsion.
Then $\nabla$ satisfies the $f$-torsion condition:
$T(fX, Y) = T(X,fY)$,
and torsion of $\nabla$ is determined by~$dF$ as
$T(X,Y,Z){=}-\frac13dF(X,Y,Z)$,
see Example 3.4 in~\cite{rst-63}. 
In~this case, \eqref{newnbl-K} reduces to
$K_XY=\frac12\,T(X,Y)$, and our Einstein's connection $\nabla$ is recovered by the equality
$\nabla_X Y=\nabla^g_X Y + \frac12\,T(X,Y)$.
\end{example}

\begin{remark}\label{R-02}
\rm
A linear connection ${\nabla}$ on a Riemannian manifold $(M,g)$ is said to be \textit{statistical}
if it is torsionless and tensor ${\nabla} g$ is symmetric in all its entries,
equivalently, the cubic form $K(X,Y,Z)$
(associated with the difference tensor of $\nabla$)
is symmetric.
The statistical structure has applications in theory of probability and statistics as well as in information geometry.
B.~Opozda \cite{opozda1} developed Bochner's technique for the statistical structure on a Riemannian manifold.
In \cite{rp-3}, we
generalized Bochner's technique to
a Riemannian manifold equipped with a regular or singular distribution and a statistical type connection.
Unfortunately, we cannot directly apply the statistical structure to Einstein's non-symmetric geometry: if the cubic form $K(X,Y,Z)$ is symmetric, then by \eqref{E-tordfnew} and~\eqref{E-tordfnew2}, $T=K=0$ is true; hence, $\nabla=\nabla^g$.
\end{remark}

For an Einstein's connection on a manifold $(M,G=g+F)$, the tensor $K(X,Y,Z)$ is not totally symmetric, see Remark~\ref{R-02}.
Thus, it is interesting to study its particular symmetries.
The~skew-symmetry of $K_X$
i.e., $K(X,Y,Z)=-K(X,Z,Y)$, by \eqref{E-tordfnew}, reduces~to
\begin{align}\label{E-cond-E2}
 T(X,Y,Z)+T(X,Z,Y) + T(X,Z,fY)+T(X,Y,fZ)=0 .
\end{align}
Using \eqref{E-cond-E2}, we find that $[K_X,K_Y]:TM\to TM$ is a skew-symmetric endomorphism:
\begin{align}\label{E-cond-KKZ}
 [K_X, K_Y](Z,W)=-[K_X, K_Y](W,Z),\quad X,Y,Z,W\in\mathfrak{X}_M.
\end{align}

\begin{example}\rm
If a manifold $(M,G=g+F)$ admits an Einstein's connection $\nabla$ with totally skew-symmetric torsion,
and the $f$-{torsion condition}
holds: $T(fX, Y) = T(X,fY)$,
see~Example~\ref{Ex-3.9a}, then \eqref{E-cond-KKZ} is true.
Indeed, for such a $\nabla$ the $f$-{torsion condition} implies
\begin{equation}\label{Eq-A-T}
 T(fY, Z) = T(Y, fZ) = -f\,T(Y, Z).
 \end{equation}
 Using \eqref{Eq-A-T} and the totally skew-symmetric torsion property in the LHS of \eqref{E-cond-E2} gives zero.
\end{example}

Define the vector fields
\begin{align*}
E:=\sum\nolimits_{\,i} K_{e_{i}}e_{i},\quad
E^*:=\sum\nolimits_{\,i} K^*_{e_{i}}e_{i},
\end{align*}
on $M$, where $K^*_X$ is the conjugate tensor to $K_X$:
$g(K^*_X Y,Z)=g(Y,K_X Z)$, and $\{e_{i}\}$ is a local orthonormal frame. Note that \eqref{E-cond-E2} yields the equality $E^*=-E$.

\subsection{The almost Lie algebroid structure}
\label{sec:algebroid}

Here, we briefly recall the construction of an almost Lie algebroid, following \cite{PMAlg,rp-2}. Lie algebroids
constitute an active field of research in differential geometry. Roughly speaking, an (almost) Lie algebroid is a structure, where one replaces the tangent bundle $TM$ of a manifold $M$ with a new smooth vector bundle
$\pi_{\cal E}\colon {\cal E}\to M$ of rank $k$, i.e.,
with a fibre $\R^k$, over $M$ with similar properties.
Many geometrical notions, which involve $TM$, were generalized to the context of Lie algebroids.

An \textit{anchor} on ${\cal E}$ is a morphism $\rho\colon {\cal E}\to TM$ of vector bundles.
 A~\textit{skew-symmetric bracket} on ${\cal E}$ is a map
$[\,\cdot,\,\cdot\,]_{\rho}\colon {\mathfrak X}_{\cal E}\times\mathfrak{X}_{\cal E}\to{\mathfrak X}_{\cal E}$ such~that
\begin{equation}\label{E-anchor}
 [Y,X]_{\rho}=-[X,Y]_{\rho},\quad
 [X, \psi Y]_{\rho}=\rho(X)(\psi)Y+\psi[X,Y]_{\rho},\quad
 \rho([X,Y]_{\rho})=[\rho(X),\rho(Y)]
\end{equation}
for all $X,Y\in{\mathfrak X}_{\cal E}$ and $\psi\in C^\infty(M)$.
The anchor and the skew-symmetric bracket give \textit{an almost Lie algebroid structure} on ${\cal E}$.
The tensor $\mathcal{J}_{\rho}: {\mathfrak X}_E\times{\mathfrak X}_{\cal E}\times{\mathfrak X}_{\cal E}\to{\mathfrak X}_{\cal E}$, given by
 $\mathcal{J}_{\rho}(X,Y,Z)=\sum\nolimits_{\,\mathrm{cycl.}}[X,[Y,Z]_{\rho}\,]_{\rho}$,
called the \emph{Jacobiator} of the bracket, prevents a bracket to satisfy the Jacobi identity.
By~\eqref{E-anchor}, third formula, we get $\rho\,\mathcal{J}_{\rho}=0$.
An almost Lie algebroid is a {Lie algebroid} provided that~$\mathcal{J}_{\rho}=0$.
%
Almost Lie algebroids are closely related to singular~distributions, i.e., given by $\rho({\cal E})$, see \cite{rp-2}.
Axiom \eqref{E-anchor}, third formula, is equivalent to vanishing of the following operator:
\begin{equation}\label{E-D-rho}
  \mathfrak{D}^\rho(X,Y) = [\rho X, \rho Y] - \rho([X,Y]_\rho).
\end{equation}
 There is a bijection between almost Lie algebroids on ${\cal E}$ and the exterior differentials of the exterior algebra
 $\Lambda({\cal E})=
 \bigoplus_{\,k\in\N}\,\Lambda^{k}({\cal E})$,
here $\Lambda^{k}({\cal E})$ is the set of $k$-forms over~${\cal E}$. The exterior differential $d^\rho$, corresponding to the almost Lie algebroid structure $({\cal E},\rho,[\,\cdot,\,\cdot\,]_\rho)$, is given~by
\begin{align}\label{E-1deg-der-rho}
\notag
 & d^\rho\,\omega(X_{0},\ldots,X_{k}) = \sum\nolimits_{\,i=0}^k(-1)^{i}(\rho X_{i})
 (\omega(X_{0},\ldots,X_{i-1},X_{i+1},\ldots,X_{k})) \\
 &+\sum\nolimits_{\,0\le i<j\le k}(-1)^{i+j}\omega
 ([X_{i},X_{j}]_\rho,X_{0},\ldots,X_{i-1},X_{i+1},\ldots,X_{j-1},X_{j+1},\ldots,X_{k}),
\end{align}
where $X_{0},\ldots,X_{k}\in{\mathfrak X}_{\cal E}$ and $\omega\in\Lambda^{k}({\cal E})$ for $k\ge0$.
For $k=0$, we have
 $d^\rho f(X)=(\rho X)(f)$, where $X\in{\mathfrak X}_{\cal E}$ and $f\in C^\infty(M)=\Lambda^{0}({\cal E})$.
Recall that a skew-symmetric bracket defines uniquely an exterior differential $d^{\rho}$ on $\Lambda(M)$,
and it gives rise to

\noindent\
-- an \textit{almost Lie algebroid} if and only if $(d^{\rho})^{2}f=0$ for $f\in C^\infty(M)$;

\noindent\
-- a \textit{Lie algebroid} if and only if $(d^{\rho})^{2}f=0=(d^{\rho})^{2}\,\omega$ for $f\in C^\infty(M)$ and $\omega\in\Lambda^{1}(M)$.

\rm
A \textit{$\rho$-connection} on $({\cal E},\rho)$ is a map $\nabla^\rho:\mathfrak{X}_{\cal E}\times\mathfrak{X}_{\cal E}\to\mathfrak{X}_{\cal E}$ satisfying Koszul~conditions
\begin{equation}\label{E-rho-conn}
 \nabla^\rho_{X}\,(fY+Z)=\rho(X)(f)Y +f\nabla^\rho_X\,Y +\nabla^\rho_{X}\,Z,\quad
 \nabla^\rho_{fX+Z}\,Y=f\nabla^\rho_{X}\,Y +\nabla^\rho_{Z}\,Y.
\end{equation}
For a $\rho$-connection $\nabla^\rho$ on ${\cal E}$, they define \textit{torsion} $T^\rho:\mathfrak{X}_{\cal E}\times\mathfrak{X}_{\cal E}\to\mathfrak{X}_{\cal E}$
and \textit{curvature} $R^\rho:\mathfrak{X}_{\cal E}\times\mathfrak{X}_{\cal E}\times\mathfrak{X}_{\cal E}\to\mathfrak{X}_{\cal E}$ by standard formulas
\begin{eqnarray*}
 T^\rho(X,Y)=\nabla^\rho_X\,Y-\nabla^\rho_Y\,X -[X,Y]_\rho,\quad
 R^\rho_{X,Y} = \nabla_{X}^\rho\nabla^\rho_{Y}
 -\nabla^\rho_{Y}\nabla^\rho_{X}
 -\nabla^\rho_{\,[X,Y]_\rho}.
\end{eqnarray*}


We will apply the construction of an almost Lie algebroid to
a manifold $(M,G=g+F)$
with an Einstein's connection, namely,
we assume ${\cal E}=TM$ and $\rho=f$,
see \eqref{E-def-f}.


\section{Main results}
\label{sec:main}

Here, motivated by the almost Lie algebroid structure,
we develop Bochner's technique for NGT spaces.
In Example~\ref{Ex-main1}, using the condition \eqref{E-condPP}, we explicitly obtain an Einstein's~connection.

\subsection{The modified covariant derivative and bracket}
\label{sec:nabla}

\begin{definition}\rm
Given an Einstein's connection $\nabla$,
on a manifold $(M, G=g+F)$,
define the map $\nabla^{f}:\mathfrak{X}_{M}\times\mathfrak{X}_{M}\to\mathfrak{X}_{M}$:
\begin{equation}\label{nablaP}
 \nabla^{f}_{X}\,Y :=\nabla^{0}_X\,Y + K_X Y,
\end{equation}
called \textit{${f}$-connection}.
The $\nabla^{f}$ depends on
the map $\nabla^{0}:\mathfrak{X}_{M}\times\mathfrak{X}_{M}\to\mathfrak{X}_{M}$ given by
\begin{align*} 
 \nabla^{0}_X\,Y := \nabla^g_{{f}X}\,Y,
\end{align*}
and the difference tensor $K$,
see \eqref{eq:contorsion-def},
but generally $\nabla^{f}$ and $\nabla^{0}$ are not linear connections on~$M$.
Set~$\nabla^{f}_{X}\,\psi=({f}X)\psi = d\psi({f}X)$
for the ${f}$-\textit{gradient} of $\psi\in C^1(M)$.
The ${f}$-derivative of an $(p,k)$-tensor $S$, where $p=0,1$,
is the following $(p,k+1)$-tensor:
\begin{align}\label{E-nablaP-S}
 (\nabla^{f} S)(Y,X_{1},\ldots,X_{k}) {=} \nabla^{f}_{Y}(S(X_{1},\ldots,X_{k}))
 {-}\sum\nolimits_{\,i=1}^{k}S(X_{1},\ldots,
 \nabla^{f}_{Y}X_{i},\ldots,X_{k}) .
\end{align}
\end{definition}

Such $\nabla^{f}$ (as well as $\nabla^{0}$) satisfies Koszul conditions,
see \eqref{E-rho-conn},
\begin{equation*}
 \nabla^{f}_{X}(\psi Y+Z)=({f}X)(\psi)Y
 +\psi\nabla^{f}_X\,Y +\nabla^{f}_{X}\,Z,\quad \nabla^{f}_{\psi X+Z}\,Y =
 \psi\nabla^{f}_{X}\,Y+\nabla^{f}_{Z}\,Y.
\end{equation*}
We use the standard notation
$\nabla^{f}_Y\,S=(\nabla^{f} S)(Y,\ldots)$.
A tensor $S$ is called $\nabla^{f}$-\textit{parallel} if $\nabla^{f}S = 0$.
Using \eqref{E-nablaP-S} and $\nabla^g\,g=0$, we obtain the following:
\begin{align}\label{Eq-proof-K}
 (\nabla^{f} g)(X, Y,Z) & =
 - K(X,Y,Z) - K(X,Z,Y) ,\\
\label{Eq-proof-KF}
 (\nabla^{f} F)(X, Y,Z) & = (\nabla^g\,F)({fX}, Y,Z) - K(X,Y,Z) - K(X,Z,Y) .
\end{align}
For an Einstein's connection $\nabla$,
using \eqref{Eq-proof-K}, \eqref{Eq-proof-KF},
and \eqref{E-tordfnew}, we calculate
\begin{align*}
 (\nabla^{f}\,g)(X, Y,Z)
 & = -\frac12 \big[T(X,Z,Y) +T(X,Y,Z) +T(X,Y,fZ)
 +T(X,Z,{f}Y)\big] ,\\
 (\nabla^{f}\,F)(X, Y,Z) &= (\nabla^g F)({fX}, Y,Z) \\
 & \quad -\frac12 \big[T(X,Z,Y) +T(X,Y,Z) +T(X,Y,fZ)
 +T(X,Z,{f}Y)\big] .
\end{align*}

The skew-symmetric ${f}$-\textit{bracket} $[\cdot,\cdot]_{f}:\mathfrak{X}_{M}\times\mathfrak{X}_{M}\rightarrow\mathfrak{X}_{M}$
is defined
by the formula
\begin{equation}\label{E-Pbracket2}
 [X,Y]_{f}:=\nabla^{f}_{X}\,Y -\nabla^{f}_{Y}\,X .
\end{equation}
By this, the $f$-connection $\nabla^{f}$ is torsionless and satisfies conditions \eqref{E-anchor} with $\rho=f$.

Following \eqref{E-D-rho} with $\rho=f$, we define the operator $\mathfrak{D}^{f}$ by the formula
\begin{equation*}
 \mathfrak{D}^{f}(X,Y) := [{f}X, {f}Y] - {f}\,[X,Y]_{f}.
\end{equation*}

\begin{proposition}
For a manifold $(M,G=g+F)$ with an Einstein's connection $\nabla$, the condition $\mathfrak{D}^{f} = 0$ is equivalent to the following:
\begin{equation}\label{E-condPP}
 (\nabla^g f)({f}X, Y) - (\nabla^g f)({f}Y, X)
 = f\,T(X,Y), \quad X,Y\in\mathcal{X}_M.
\end{equation}
\end{proposition}

\begin{proof}
Put $[X,Y]_{0}:=\nabla^{0}_{X}\,Y -\nabla^{0}_{Y}\,X$.
Using \eqref{E-Pbracket2}, \eqref{nablaP} and
\eqref{E-tordfnew2}, we have
\begin{align}\label{E-Pbracket-b}
\notag
 [X,Y]_{f}
 & =\nabla^g_{fX}Y-\nabla^g_{fY}X +K_XY-K_YX
 =\nabla^g_{fX}Y-\nabla^g_{fY}X +T(X,Y) \\
 & = [X,Y]_{0} + T(X,Y).
\end{align}
Using \eqref{E-Pbracket-b}, we find
\begin{eqnarray*}
 \mathfrak{D}^{f}(X,Y) \eq \nabla^g_{fX}(fY)
 -f\nabla^g_{fX}Y
 -\nabla^g_{fY}(fX) +f\nabla^g_{fY}X -f\,T(X,Y) \\
 \eq (\nabla^g f)({fX}, Y)-(\nabla^g f)({fY}, X)
 -f\,T(X,Y),
\end{eqnarray*}
and the required \eqref{E-condPP} follows.
\end{proof}

\begin{example}\label{Ex-main1}\rm
 If condition \eqref{E-condPP} is true, then we obtain an almost Lie algebroid structure on $M$
 with an anchor $f:TM\to M$ and a skew symmetric bracket $[\,\cdot,\,\cdot\,]_{f}$ on ${\cal E}=TM$ satisfying \eqref{E-anchor}.
 On the other hand, from \eqref{E-condPP}, we obtain the torsion $T$, hence, an example of the Einstein's connection~$\nabla$.
\end{example}






For an ${f}$-connection $\nabla^{f}$ on $(M,G=g+F)$, its \textit{dual ${f}$-connection} $\breve\nabla^{f}$ is defined~by
\[
 {f}X(g(Y,Z)) = g(\nabla^{f}_X Y, Z) + g(Y, \breve\nabla^{f}_X Z).
\]
In~turn, the ${f}$-connection $\nabla^{f}$ is dual to $\breve\nabla^{f}$.
One may show that the following equality:
\begin{equation}\label{E-brave-nabla}
 \breve\nabla^{f}_X Y = \nabla^{0}_{X}Y - K^*_X Y
\end{equation}
 is true in general.
For a dual ${f}$-connection $\breve\nabla^{f}$, we define the ${f}$-bracket and $\breve{\mathfrak{D}}^{f}$~by
\[
 {[X,Y]}^\smile_{f} :=\breve\nabla^{f}_{X}\,Y-\breve\nabla^{f}_{Y}\,X,\quad
 \breve{\mathfrak{D}}^{f}(X,Y) :=
 [{f}X, {f}Y] - {f}{[X,Y]}^\smile_{f}.
\]
If \eqref{E-cond-E2} holds,
then using \eqref{E-brave-nabla}, we get
$\breve\nabla^{f}_X Y = \nabla^{f}_{X}Y$; hence, the following is true:
\begin{equation}\label{E-bracket-b}
 {[X,Y]}^\smile_{f}=[X,Y]_{f},\quad
 \breve{\mathfrak{D}}^{f} = \mathfrak{D}^{f}.
\end{equation}
For any $k$-form $\omega$, set
\begin{equation}\label{E-stat-L1-0}
 (K_{Y}\,\omega)(X_1,X_2,\ldots,X_k)
 := -\sum\nolimits_{\,i} \omega(X_{1},\ldots,K_{Y}X_{i},\ldots,X_{k}).
\end{equation}

The following lemma is similar to Lemmas~6.2 and 6.3 in \cite{opozda1}.

\begin{lemma}
Let $\nabla$ be an Einstein's connection on a manifold $(M,G=g+F)$,
$\omega\in\Lambda^k(M)$ and $\{e_{i}\}$ a~local orthonormal frame.
If condition \eqref{E-cond-E2} is true, then we have
\begin{equation}\label{E-stat-L1-b}
 \sum\nolimits_{\,i}\iota_{\,e_i}((K_{e_{i}}+K^*_{e_{i}})\,\omega) = 0 ,
\end{equation}
where $\iota$ stands for the interior product.
\end{lemma}

\begin{proof}
We have
\begin{align*}
 \sum\nolimits_{\,i}(
 (K_{e_{i}}+K^*_{e_{i}})
 \,\omega)(e_{i},X_2,\ldots,X_k) =&
 -\omega(\sum\nolimits_{\,i}
 (K_{e_{i}}+K^*_{e_{i}})
 e_i, X_2,\ldots,X_k)\\
 & -\sum\nolimits_{\,i}\omega(e_i, X_2,\ldots,
 (K_{e_{i}}+K^*_{e_{i}}) X_p, \ldots,X_k) .
\end{align*}
 If $k=1$, the last term in the above does not appear.
 If $k>1$, we fix $2\le p\le k$ and~find
 \begin{align*}
 & \sum\nolimits_{\,i} \omega(e_i, X_2,\ldots,
 (K_{e_{i}}+K^*_{e_{i}})X_p,\ldots,X_k) \\
 & = \sum\nolimits_{\,i<j}
 \big\{
 g((K_{e_{i}}+K^*_{e_{i}})X_p, e_j)
 - g((K_{e_{j}}+K^*_{e_{j}})X_p
 \big\}
 \omega(e_i, X_2,\ldots, e_j,\ldots,X_k).
\end{align*}
Using \eqref{E-tordfnew} and the assumption \eqref{E-cond-E2} completes the proof of \eqref{E-stat-L1-b}.
\end{proof}

\subsection{The ${f}$-divergence and Bochner ${f}$-Laplacian}
\label{sec:div}


Define the ${f}$-\textit{divergence} of a vector field $X$ on a manifold $(M,G=g+F)$ by
\begin{equation}\label{E-Pdiv-1}
 \Div_{f}X = \operatorname{trace}(Y{\to}\,\nabla^{f}_{Y}\,X).
\end{equation}
The divergence of a $(1,k)$-tensor $S$ is defined by
$\Div_g S =\operatorname{trace}(Y\to\nabla^g_{Y} S)$, that is
\begin{equation}\label{eq:div2}
 (\Div_g S)(X_1,\ldots, X_k) = \sum\nolimits_{\,i}
 g((\nabla^g\,S)({e_i}, X_1,\ldots, X_k), e_i) ,
\end{equation}
using a local orthonormal frame $\{e_{i}\}$.
The divergence of a vector field $X$ can also be represented as follows:
\begin{equation}\label{E-divf-1}
 d(\iota_{X}\,d\operatorname{vol}_g) = ({\rm div}_g\,X )\,d\operatorname{vol}_g.
\end{equation}

The ``musical" isomorphisms $\sharp:T^*M\to TM$ and $\flat: TM\to T^*M$ are used for rank one tensors, e.g. if $\omega\in \Lambda^1(M)$
and $X\in {\mathfrak X}_M$ then $\omega(X)=g(\omega^\sharp,X)=g(\omega,X^\flat)=X^\flat(\omega^\sharp)$.
In order to generalize the Stokes Theorem,
 we formulate the following.

\begin{lemma}\label{L-03B}
On a manifold $(M,G=g+F)$ with an Einstein's connection $\nabla$,
the condition
\begin{equation}\label{E-cond-PP-stat}
 \Div_g {f} = (E^*)^\flat
\end{equation}
is equivalent to the following equality:
\begin{equation}\label{E-divf-4B}
 \Div_{f} X=\Div_g({f}X),\quad X\in{\mathcal X}_{M}.
\end{equation}
\end{lemma}

\begin{proof}  Using
$fe_i=\sum\nolimits_{\,j}g({f}e_i, e_j)e_j$
and \eqref{eq:div2} with $S=f$,
we obtain
\begin{eqnarray*}
 && \sum\nolimits_{\,i}g(\nabla^g_{{f}e_i} X, e_i) = \sum\nolimits_{\,i,j}g({f}e_i, e_j)g(\nabla^g_{e_j} X, e_i) =
 \sum\nolimits_{\,i,j}g(e_i, -{f}e_j)g(\nabla^g_{e_j} X, e_i) \\
 && =\sum\nolimits_{\,j}g(\nabla^g_{e_j} X, -{f}e_j) = \sum\nolimits_{\,j}g({f}\nabla^g_{e_j} X, e_j)
 = \Div_g({f}X) - (\Div_g{f})(X).
\end{eqnarray*}
Using this and the definition \eqref{nablaP}, we have
\begin{equation*}
 \Div_{f} X = \sum\nolimits_{\,i}\,g(\nabla^g_{{f}e_i} X + K_{e_i}X, e_i) = \Div_g({f}X) - (\Div_g{f})(X)
 + g(X,E^*).
\end{equation*}
From this and \eqref{E-cond-PP-stat},
the required statement follows.
\end{proof}


Lemma~\ref{L-03B} allows us to extend the Divergence Theorem
to NGT spaces.

\begin{theorem}\label{T-P-Stokes}
Let an Einstein's connection $\nabla$ on a compact manifold $(M,G=g+F)$ with boun\-dary satisfy \eqref{E-cond-PP-stat}.
Then for any vector field $X\in\mathfrak{X}_M$  we have
\begin{equation}\label{E-Div-Th-P1}
 \int_M (\Div_{f}X)\,d\operatorname{vol}_{g} = \int_{\partial M}g(X,\,{f}\nu)\,\omega_{\,\hat g},
\end{equation}
where
$\nu$ is the unit inner normal to $\partial M$.
On a manifold $(M,G=g+F)$ without boundary, for any $X\in\mathfrak{X}_M$ with compact support,
\eqref{E-Div-Th-P1} reduces to the equality
 $\int_M (\Div_{f} X) d\operatorname{vol}_{g} {=} 0$.
\end{theorem}

The following pointwise inner products and norms for $(0,k)$-tensors will be used:
\[
 g(S_1,\,S_2) = \sum\nolimits_{\,i_1,\ldots,\,i_k} S_1(e_{i_1},\ldots,\,e_{i_k})\,S_2(e_{i_1},\ldots,\,e_{i_k}),\quad
  \|S\| = \sqrt{g(S,\,S)},
\]
while, for $k$-forms, we set
 $g(\omega_1,\,\omega_2) = \sum\nolimits_{\,i_1<\ldots\,<i_k} \omega_1(e_{i_1},\ldots,\,e_{i_k})\,\omega_2(e_{i_1},\ldots,\,e_{i_k})$.

For $L^2$-product of compactly supported tensors on a Riemannian manifold, we set
\[
 (S_1,\,S_2)_{L^2}=\int_M g(S_1,\,S_2)\,
 d\operatorname{vol}_{g}.
\]
The~$L^2$-adjoint operator $\nabla^{*{f}}$
(and similarly $\breve\nabla^{*{f}}$)
for $\nabla^{f}$
maps $(p,k+1)$-tensors, where $p=0,1$, to $(p,k)$-tensors:
\begin{equation}\label{E-nabla-astP}
 (\nabla^{*{f}} S)(X_{1},\ldots,X_{k}) =
 -\sum\nolimits_{\,i} (\nabla^{f}_{e_{i}} S)(e_{i},X_{1},\ldots,X_{k}) .
\end{equation}

\begin{lemma}\label{L-03}
The $\nabla^{*{f}}$ and $\breve\nabla^{*{f}}$ are related for any form $\omega\in\Lambda^k(M)$ by
\begin{align}\label{E-breve-nabla}
 \breve\nabla^{*{f}}\omega = \nabla^{*{f}}\omega
 +\sum\nolimits_{\,i}\iota_{\,e_i}((K_{e_i}+K^*_{e_i})\,\omega);
\end{align}
moreover, if \eqref{E-cond-E2} is true, then
$\breve\nabla^{*{f}}\omega = \nabla^{*{f}}\omega$.
\end{lemma}

\begin{proof}
Using \eqref{E-nabla-astP} and \eqref{E-nablaP-S},
we relate $\nabla^{*{f}}$ and $\breve\nabla^{*{f}}$ with $\nabla^{*{0}}$ for any form $\omega\in\Lambda^k(M)$:
\begin{equation*}
 \nabla^{*{f}}\,\omega = \nabla^{*{0}}\,\omega
 +\sum\nolimits_{\,i}\iota_{\,e_i}(K_{e_i}\,\omega),
 \quad
 \breve\nabla^{*{f}}\,\omega = \nabla^{*{0}}\,\omega
 -\sum\nolimits_{\,i}\iota_{\,e_i}(K^*_{e_i}\,\omega).
\end{equation*}
Hence \eqref{E-breve-nabla} is true.
Applying \eqref{E-cond-E2} and \eqref{E-stat-L1-b} to \eqref{E-breve-nabla}, yields
$\breve\nabla^{*{f}}\omega = \nabla^{*{f}}\omega$.
\end{proof}

The $\nabla^{*{f}}$ is related to the ${f}$-\textit{divergence} \eqref{E-Pdiv-1}~by
\begin{equation}\label{E-divPX}
 \Div_{f}X  = -\nabla^{*{f}} X^\flat ,\quad
 X\in{\cal X}_{M}.
\end{equation}

The next result shows when $\nabla^{*{f}}$ is $L^2$-{adjoint to} $\nabla^{f}$ on $k$-forms.


\begin{proposition}
Let an Einstein's connection $\nabla$ satisfy
\eqref{E-cond-E2} and \eqref{E-cond-PP-stat}.
Then for any compactly supported form $\omega_1\in\Lambda^k(M)$
and
a form $\omega_2\in\Lambda^{k+1}(M)$, we~have
\begin{equation}\label{E-cond-PP-intB}
 (
 \nabla^{*{f}} \omega_2,\ \omega_1)_{L^2} = (\omega_2,\ \nabla^{f} \omega_1)_{L^2}.
\end{equation}
\end{proposition}

\begin{proof}
Define a compactly supported 1-form $\omega$ by
 $\omega(Y) = g(\iota_{\,Y}\,\omega_2,\,\omega_1)$, where $Y\in{\mathcal X}_{M}$.
To~simplify calculations,
assume that $k=1$.
Take an orthonormal frame $(e_{i})$ such that $\nabla_Y e_{i}=0$ for $Y\in T_xM$.
To simplify calculations, assume that $s=t=1$, then 
at~$x\in M$,
\begin{align*}
 & -\nabla^{*{f}}\omega =\sum\nolimits_{j}(\nabla^{f}_{e_j}\omega)(e_j)
 = \sum\nolimits_{i,j} g({f} e_j, e_i)\,(\nabla_{e_i}\omega)(e_j)\\
 & = \sum\nolimits_{i,j,c} g({f} e_j, e_i)\,\nabla_{e_i} g(\omega_2(e_j,e_{c}), \omega_1(e_{c})\,)\\
 & = \sum\nolimits_{i,j,c} g({f} e_j, e_{i})
 \Big[ g(\nabla_{e_i} \omega_2(e_j,e_{c}), \omega_1(e_{c})\,) + g(\omega_2(e_j,e_{c}), \nabla_{e_i} \omega_1(e_{c})\,)\Big] \\
 &= \sum\nolimits_{i,j,c} \Big[ g(\,g({f} e_j, e_i)\nabla_{e_i}\omega_2(e_j,e_{c}), \omega_1(e_{c})) +g(\omega_2(e_j,e_{c}),g({f} e_j, e_i)\nabla_{e_i}\omega_1(e_{c})\,)\Big] \\
 &= \sum\nolimits_{j,c}\Big[
 g(\nabla^{f}_{e_j}\omega_2(e_j,e_{c}), \omega_1(e_{c})\,) +g(\omega_2(e_j,e_{c}), \nabla^{f}_{e_j}\omega_1(e_{c})\,)\Big] \\
 &= -g(\nabla^{*{f}}\omega_2, \omega_1)
 + g(\omega_2, \nabla^{f} \omega_1).
\end{align*}
By the above equality, using \eqref{E-breve-nabla},
we obtain
\begin{equation}\label{E-div-f-omega}
 -\nabla^{*{f}}\omega = -g(\breve\nabla^{*{f}} \omega_2,\ \omega_1) + g(\omega_2,\ \nabla^{f} \omega_1)
 -g(\iota_{E+E^*}\,\omega_2,\,\omega_1).
\end{equation}
By \eqref{E-cond-E2}, $E+E^*=0$ and $\breve\nabla^{*{f}}=\nabla^{*{f}}$ hold.
Thus, by \eqref{E-divPX} and Theorem~\ref{T-P-Stokes} with $X^\flat=\omega$, we get~\eqref{E-cond-PP-intB}.
\end{proof}


The
operator
$\nabla^{*{f}}\,\nabla^{f}$ will be called the \textit{Bochner ${f}$-Laplacian} for an ${f}$-connection.


\begin{proposition}\label{P-max}
Let an Einstein's connection $\nabla$ on a closed manifold $(M,G=g+F)$ satisfy \eqref{E-cond-E2} and \eqref{E-cond-PP-stat}.
If~$g(
\nabla^{*f}\,\nabla^{f}\omega,\,\omega)\le0$ holds for a form $\omega\in\Lambda^k(M)$, then
$\nabla^{f} \omega=0$.
\end{proposition}

\begin{proof}
We apply \eqref{E-cond-PP-intB},
\[
 0\ge(
 \nabla^{*{f}}\,\nabla^{f}\omega,\,\omega)_{L^2}=(\nabla^{f}\omega,\,\nabla^{f}\omega)_{L^2}\ge 0;
\]
hence, $\omega$ is ${f}$-parallel: $\nabla^{f} \omega=0$.
\end{proof}

\subsection{The Hodge ${f}$-Laplacians}
\label{sec:laplace}

Using an Einstein's connection, define the \textit{exterior ${f}$-derivative}
of
$\omega\in\Lambda^{k}(M)$
by
\begin{equation}\label{E-dP}
 d^{f}\omega(X_{0},\ldots,X_{k})
 =\sum\nolimits_{\,i}(-1)^{i}(\nabla^{f}\omega)(X_{i}, X_{0},\ldots,X_{i-1},X_{i+1},\ldots X_{k}),
\end{equation}
and similarly for $\breve d^{\,{f}}$.
For $\psi\in C^\infty(M)$ and $X\in\mathfrak{X}_M$, we have $d^{f}\psi(X)=\nabla^{f}_X\psi$ and
$\breve d^{\,{f}}\psi(X)=\breve\nabla^{f}_X\psi$.

\begin{proposition}
The ${f}$-derivative $d^{f}:{\mathrm\Omega}^{k}(M)\rightarrow{\mathrm\Omega}^{k+1}(M)$ is a 1-degree derivation, see \eqref{E-1deg-der-rho}:
\begin{align}\label{E-1deg-der}
\nonumber
 & d^{f}\omega(X_{0},\ldots,X_{k}) =\sum\nolimits_{\,i=0}^{k}(-1)^{i}{f}X_{i}(\omega(X_{0},\ldots, X_{i-1},X_{i+1},\ldots,X_{k}))\\
 & +\sum\limits_{\,0\le i<j\le k}(-1)^{i+j}\omega
 \big([X_i, X_j]_{f}, X_{0},\ldots, X_{i-1},X_{i+1},\ldots,
 X_{j-1},X_{j+1},\ldots,X_{k}\big).
\end{align}
Similarly to \eqref{E-1deg-der}, for any $\omega\in\Lambda^{k}(M)$, we~get
\begin{align}\label{E-1deg-derB}
\notag
 & {\breve d}^{f}\omega(X_0,\ldots,X_k) =\sum\nolimits_{\,i=0}^{k}(-1)^{i} {f} X_{i}(\omega(X_{0},\ldots,
 X_{i-1},X_{i+1},\ldots,X_{k}))\\
 & +\sum\limits_{\,0\le i<j\le k}(-1)^{i+j}\omega
 \big([X_i, X_j]^\smile_{f}, X_0,\ldots, X_{i-1},X_{i+1},\ldots,
 X_{j-1},X_{j+1},\ldots,X_k\big).
\end{align}
Moreover, if condition \eqref{E-cond-E2} is valid, then
${\breve d}^{f}\omega = d^{f}\omega$.
\end{proposition}

\begin{proof}
For any form $\omega\in\Lambda^{k}(M)$, using \eqref{E-dP},
we obtain
\begin{align*}
 & d^{f}\omega(X_{0},\ldots,X_{k})
 =\sum\nolimits_{\,i=0}^{k}(-1)^{i}{f}X_{i}(\omega(X_{0},\ldots,X_{i-1},X_{i+1},\ldots,X_{k}))\\
 & +\sum\nolimits_{\,i=0}^{k}(-1)^{i}\Big(\sum\nolimits_{\,j=0}^{i-1}
\omega(X_{0},\ldots,\nabla^{f}_{X_{i}}X_{j},\ldots,X_{i-1},X_{i+1},\ldots,X_{k})\\
 & +\sum\nolimits_{\,j=i+1}^{k}\omega(X_{0},\ldots,X_{i-1},X_{i+1},\ldots,\nabla^{f}_{X_{i}}X_{j},\ldots,X_{k})\Big)\\
 & =\sum\nolimits_{\,i=0}^{k}(-1)^{i}{f}X_{i}(\omega(X_{0},\ldots,X_{i-1},X_{i+1},\ldots,X_{k}))\\
 & +\sum\limits_{\,0\le i<j\leq k}(-1)^{i+j}\omega\big(\nabla^{f}_{X_{i}}X_{j}
 -\nabla^{f}_{X_{j}}X_{i},X_{0},\ldots,X_{i-1},X_{i+1},\ldots,X_{j-1},X_{j+1},\ldots,X_{k}\big).
\end{align*}
Applying \eqref{E-Pbracket2} to the above, we complete the proof
of \eqref{E-1deg-der}.
The proof of \eqref{E-1deg-derB} is similar.
The~last statement follows from the above and the first equality of \eqref{E-bracket-b}.
\end{proof}

Define the
${f}$-\textit{codifferential} $\delta^{f}:\Lambda^{k}(M) \to\Lambda^{k-1}(M)$ by
 $\delta^{\,f}:=\nabla^{*{f}}$.
Similarly, we set $\breve\delta^{\,f}:=\breve\nabla^{*{f}}$.
Using the analog of \eqref{E-nabla-astP}, we obtain
\[
 \breve\delta^{\,{f}}\omega(X_{2},\ldots,X_{k})
  =-\sum\nolimits_{\,i}(\breve\nabla^{f}_{e_{i}}\,\omega)(e_{i},X_{2},\ldots,X_{k}).
\]

\begin{proposition}\label{P-dP-deltaPB}
On a closed manifold $(M,G=g+F)$ with an Einstein's connection satis\-fying
\eqref{E-cond-E2} and \eqref{E-cond-PP-stat},
for any
$\omega_{1}\in\Lambda^{k}(M)$
and $\omega_{2}\in\Lambda^{k+1}(M)$, we~have
 $(\breve\delta^{\,{f}}\omega_{2},\,\omega_{1})_{L^2} =(\omega_{2},\,d^{f}\omega_{1})_{L^2}$.
\end{proposition}

\begin{proof} We derive as in the classical case:
\begin{align*}
 & g(d^{f}\omega_{1},\,\omega_{2}) = \\
 & = \sum\nolimits_{\,u=0}^{k}(-1)^{i}\,
 \nabla^{f}_{\partial_{i_{u}}}\omega_{1}(\partial_{i_{0}},\ldots,
 \partial_{i_{u-1}}, \partial_{i_{u+1}},
 \ldots,\partial_{i_{k}}) g^{i_{0}j_{0}}\ldots g^{i_{k}j_{k}} \omega_{2}(\partial_{j_{0}},\ldots,\partial_{j_{k}}) \\
 & = (k+1)\big(\nabla^{f}_{\partial_{i_{0}}}\omega_{1}(\partial_{i_{1}},\ldots,\partial_{i_{k}})\big)
 g^{i_{0}j_{0}}\ldots g^{i_{k}j_{k}}\omega_{2}(\partial_{j_{0}},\ldots,\partial_{j_{k}})
 =g(\nabla^{f}\omega_{1},\,\omega_{2}) .
\end{align*}
It appears a $(k + 1)$ factor, that finally is absorbed in the definition of~$d^{f}$.
Similarly,
we get
$g(\breve d^{f}\omega_{2},\,\omega_{1}) =g(\breve\nabla^{*{f}}\omega_{2},\,\omega_{1})$.
By the above, we have for the vector field $X=\omega^\flat$ (dual to the 1-form $\omega(Y)=g(\iota_{\,Y}\,\omega_2,\,\omega_1)$),
an analog of \eqref{E-div-f-omega}:
\begin{equation*}
 \Div_f X =
 -g(\breve d^{f} \omega_2,\,\omega_1)
 + g(\omega_2,\,d^{f}\omega_1)
 +\sum\nolimits_{\,i}\iota_{\,e_i}((K_{e_i}+K^*_{e_i})\,\omega).
 \end{equation*}
Using this,
\eqref{E-stat-L1-b} and Theorem~\ref{T-P-Stokes},
completes the proof.
\end{proof}


\begin{definition}\rm
Define the \textit{Hodge ${f}$-Laplacians} ${\mathrm\Delta}_H^{\,{f}}$ and $\breve{\mathrm\Delta}^{f}_H$ for differential forms by
\begin{eqnarray*}
 {\mathrm\Delta}_H^{\,{f}} = d^{f}\breve\delta^{\,{f}} +\breve\delta^{\,{f}} d^{f},\quad
 \breve{\mathrm\Delta}^{f}_H = d^{f}\delta^{f} + \delta^{f} d^{f}.
\end{eqnarray*}
A
form $\omega$ is called
${f}$-\textit{harmonic} if
${\mathrm\Delta}_H^{\,{f}}\,\omega=0$ and
$\|\omega\|_{L^2}<\infty$ (similarly for $\breve{\mathrm\Delta}^{f}_H$).
\end{definition}

The ${f}$-harmonic forms have similar properties as in the classical case, e.g., \cite{Peter}.
Let \eqref{E-divf-4B} and \eqref{E-cond-E2}
hold on a closed manifold $(M,G=g+F)$,
then for $\omega\in\Lambda^k(M)$, by Proposition~\ref{P-dP-deltaPB},
we~have
 $({\mathrm\Delta}^{f}_H\,\omega,\,\omega)_{L^2} =
 ( d^{f}\omega,\,d^{f}\omega)_{L^2}
 + (\breve\delta^{\,{f}}\omega,\,\breve\delta^{\,{f}}\omega)_{L^2}$;
such $\omega$ is ${f}$-harmonic (and similarly for $\breve{\mathrm\Delta}^{f}_H$-harmonic case) if and only if
$d^{f}\omega=0$ and $\breve\delta^{\,{f}}\omega=0$.
If~${\mathrm\Delta}^{f}_H\,\omega=0$ and
$\omega=d^{f}\,\theta$ hold
for $\omega\in\Lambda^k(M)$, then
$\breve\delta^{\,{f}} d^{f}\theta=\breve\delta^{\,{f}}\omega=0$; thus,
\[
 (\omega,\,\omega)_{L^2}
 =(d^{f}\theta,\, d^{f}\theta)_{L^2}
 =(\theta,\,\breve\delta^{\,{f}} d^{f}\theta)_{L^2}
 = (\theta,\,\breve\delta^{\,{f}}\omega)_{L^2} =0;
\]
and if $\omega\in\Lambda^k(M)$
is ${f}$-harmonic and $\omega=d^{f}\,\theta$ for some $\theta\in\Lambda^{k-1}(M)$, then $\omega=0$.


We also consider the Hodge type Laplacian related to $\nabla^{0}$, defined by
 $\Delta^{0}_H = \delta^{\,0}\,d^{\,0}+d^{\,0}\,\delta^{\,0}$,
where for any differential form $\omega\in\Lambda^{k}(M)$,
\begin{align*}
 & d^{\,0}\omega(X_{0},\ldots,X_{k})=\sum\nolimits_{\,i}(-1)^{i}(\nabla^g\omega)(fX_{i}, X_{0},\ldots,X_{i-1},X_{i+1},\ldots X_{k}),\\
 & \delta^{\,0}\omega(X_{2},\ldots,X_{k}) =-\sum\nolimits_{\,i}({\nabla}^g\omega)(fe_{i}, e_{i},X_{2},\ldots,X_{k}).
\end{align*}

\begin{lemma} For an Einstein's connection $\nabla$
and any form $\omega\in\Lambda^{k}(M)$ we have
\begin{align}\label{E-58-59A}
\delta^{\,{0}}\omega =\delta^{f}\omega
 +\sum\nolimits_{\,i}\iota_{\,e_i}(K_{e_i}\omega)
 =\breve\delta^{f}\omega
 -\sum\nolimits_{\,i}\iota_{\,e_i}(K^*_{e_i}\omega).
\end{align}
Moreover, if condition \eqref{E-cond-E2} is valid, then
\begin{align}\label{E-58-59}
\breve\delta^{f}\omega = \delta^{f}\omega, \quad
\breve{\mathrm\Delta}^{f}_H\,\omega={\mathrm\Delta}^{f}_H\,\omega.
\end{align}
\end{lemma}

\begin{proof}

Using \eqref{E-nablaP-S}, \eqref{E-stat-L1-0} and \eqref{E-nabla-astP}, we obtain the first equality in \eqref{E-58-59A}:
\begin{equation*}
\delta^{f}\omega =
-\sum\nolimits_{\,i}\nabla^{f}_{e_{i}}\iota_{e_{i}} \omega
=\delta^{\,{0}}\omega
-\sum\nolimits_{\,i}\iota_{e_{i}}(K_{e_{i}}\omega).
\end{equation*}
Similarly, we get the second equality in \eqref{E-58-59A}.
This and \eqref{E-stat-L1-b}
yield the equalities \eqref{E-58-59}.
\end{proof}

We will modify Stokes theorem on a complete open Riemannian manifold.

\begin{proposition}\label{L-Div-1}
Let a complete noncompact manifold $(M,G=g+F)$ be endowed with Einstein's connection $\nabla$ and
a vector field $X$ such that $\Div_{f}X\ge0$ or $\Div_{f}X\le0$, and let conditions \eqref{E-cond-PP-stat} and $\|{f}X\|_{g}\in\mathrm{L}^{1}(M,g)$ hold.
Then $\Div_{f}X\equiv0$.
\end{proposition}

\begin{proof}
Let $\omega$ be the $(n-1)$-form in $M$ given by $\omega= \iota_{{f}X}\,d\operatorname{vol}_{g}$, i.e., the contraction
of the volume form $d\operatorname{vol}_{g}$ in the direction of ${f}X$.
If $\{e_{1}, \ldots, e_{n}\}$ is an orthonormal frame on an open set $U\subset M$ with coframe ${\omega_{1}, \ldots, \omega_{n}}$, then
\[
 \iota_{{f}X}\,d\operatorname{vol}_{g} = \sum\nolimits^{n}_{i=1} (-1)^{i-1}
 g({f}X, e_{i})\,\omega_{1}\wedge\ldots\wedge\omega_{i-1}\wedge\omega_{i+1}\wedge\ldots\wedge\omega_{n}.
\]
Since the $(n - 1)$-forms $\omega_{1}\wedge\ldots\wedge\omega_{i-1}\wedge\omega_{i+1}\wedge\ldots\wedge\omega_{n}$
are orthonormal in ${\mathrm\Omega}^{n-1}(M)$, we get $\|\omega\|_{g}^{2} = \sum\nolimits^{n}_{i=1} g({f}X, e_{i})^{2} =\|{f}X\|_{g}^{2}$.
Thus, $\|\omega\|_{g}\in\mathrm{L}^{1}(M,g)$ and
 $d\omega
 (\Div_{f}X)\,d\operatorname{vol}_{g}$,
see~\eqref{E-divf-1} and~\eqref{E-divf-4B}. There exist
domains $B_{i}$ on $M$ such that
 $M=\bigcup\nolimits_{\,i\ge1} B_{i}$,
 $B_{i}\subset B_{i+1}$,
 and
 $\lim_{\,i\to\infty}\int\nolimits_{B_{i}} d\omega=0$.
Then
\[
 \int_{B_{i}} (\Div_{f}X)\,d\operatorname{vol}_{g}
 \overset{\eqref{E-divf-4B}}
 =\int_{B_{i}}\Div_g({f}X)\,d\operatorname{vol}_{g}
 =\int_{B_{i}} d\omega\to0.
\]
But since $\Div_{f}X \ge 0$ or $\Div_{f}X\le0$
on $M$, it follows that $\Div_{f}X = 0$ on $M$.
\end{proof}

We call ${\mathrm\Delta}^{f}\psi = \Div_{f}(\nabla^{f}\psi)$ the ${f}$-\textit{Laplacian for a function $\psi\in C^2(M)$}.
By \eqref{nablaP},
 ${\mathrm\Delta}^{f}\psi = {\mathrm\Delta}^{0}\psi + ({f}E)(\psi)$.



The next theorem extends the
classical result on subharmonic functions.

\begin{theorem}
Let a manifold $(M,G=g+F)$ be complete noncompact, and \eqref{E-cond-PP-stat} hold for an Einstein's connection $\nabla$.
If $\|{f}\nabla^{f}\psi\|_g$ and $\|\psi{f}\nabla^{f}\psi\|_g$ belong to $L^1(M,g)$
for some $\psi\in C^2(M)$ satisfying either ${\mathrm\Delta}^{f}\psi\ge0$ or ${\mathrm\Delta}^{f}\psi\le0$,
then $\nabla^{f}\psi=0$; moreover, if either $f(TM)=TM$ ($\dim M$ is even), or $f(TM)\ne TM$
and ${f}(TM)$ is bracket-generating, then $\psi\!=\const$.
\end{theorem}

\begin{proof} Set $X=\nabla\,^{f}\psi$, then $\Delta^{f}\,\psi=\Div_{f}X$.
By Proposition~\ref{L-Div-1} with $X=\nabla\,^{f}\psi$ and condition $\|{f}\nabla^{f}\psi\|\in L^1(M,g)$,
we get $\Delta\,^{f}\,\psi\equiv0$.
Using
 $\Div_{f}(\psi\cdot Y) = \psi\cdot\Div_{f}\,Y + g(\nabla^{f}\psi,\ Y)$
with $Y=\nabla\,^{f}\psi$,
Proposition~\ref{L-Div-1} with $X=\psi\nabla^{f} \psi$ and condition $\|\psi\,{f}\nabla\,^{f}\psi\|\in L^1(M,g)$, we get
$(\nabla^{f}\psi,\,\nabla^{f}\psi)_{L^2}=0$, hence $\nabla^{f}\psi= 0$.
If $f(TM)=TM$, then $\psi\!=\const$.
If $f(TM)\ne TM$, then
$\nabla^{f}\psi=0$ means that $\psi$ is constant along the curves in $f(TM)$; moreover, if the singular distribution $f(TM)$ is bracket-generating then, by Chow's theorem \cite{Chow}, $\psi={\rm const}$. 
\end{proof}


%


\subsection{The $f$-curvature operator}
\label{sec:R}

\begin{definition}\rm
Define the \textit{second ${f}$-derivative} of an $(p,k)$-tensor $S$ as the $(p,k+2)$-tensor
\begin{equation*}
 (\nabla^{f})^2_{\,X,Y}\,S := \nabla^{f}_{X}\,(\nabla^{f}_{Y}\,S)
 -\nabla^{f}\,\!\!_{\,\nabla^{f}_{X}Y}\,S,\quad X,Y\in\mathfrak{X}_M.
\end{equation*}
Define the ${f}$-\textit{curvature operator} of $\nabla^{f}$ by
standard formula, see Section~\ref{sec:algebroid},
\begin{align*}
 R^{f}_{X,Y} := (\nabla^{f})^2_{\,X,Y}
 -(\nabla^{f})^2_{\,Y,X}
 = \nabla_{X}^{f}\,\nabla^{f}_{Y} -\nabla^{f}_{Y}\,\nabla^{f}_{X}
 -\nabla^{f}_{\,[X,Y]_{f}},\quad X,Y\in\mathfrak{X}_M,
\end{align*}
and set
\begin{equation}\label{Curv-P}
  {R}^{f}(X,Y,Z,W) = g({R}^{f}_{X,Y}Z,\,W),\quad X,Y,Z,W\in\mathfrak{X}_M.
\end{equation}
The ${f}$-\textit{Ricci curvature operator} of $\nabla^{f}$ is defined by the standard way:
\begin{equation}\label{E-Ric0}
 {\rm Ric}^{f} X = \sum\nolimits_{\,i}
 R^{f}_{\,X,e_i}\,e_i,\quad
 {\rm Ric}^{f}(X,Y) = \sum\nolimits_{\,i}
 R^{f}(X,e_i,e_i,Y)
 = g({\rm Ric}^{f} X, Y).
\end{equation}
\end{definition}

The action of $R^{f}$ on $(0,k)$-tensor fields is defined similarly to the action of the Riemannian curvature tensor $R^g$:
\begin{equation}\label{Curv-S}
  (R^{f}_{X,Y}\,S)(X_1,\ldots,X_k) ={\mathfrak{D}}^{f}(X,Y)(S(X_1,\ldots,X_k))
  -\sum\nolimits_{\,i}S(X_1,\ldots R^{f}_{X,Y}X_i,\ldots,X_k).
\end{equation}
To simplify the calculations, define the following tensor:
\begin{align*}
 \Theta_{X,Y} := (\nabla^g_{{f}X} K)_Y - (\nabla^g_{{f}Y} K)_X +[K_X,K_Y]-K_{[X,Y]_f} ,\quad X,Y\in\mathfrak{X}_M.
\end{align*}
Here, $(\nabla^g_{{f}X}\,K)_YZ=\nabla^g_{{f}X}(K_Y Z) -K_{\nabla^g_{{f}X}Y}Z -K_Y(\nabla^g_{{f}X}Z)$.
If
$K(X,Y,Z)=-K(X,Z,Y)$, see \eqref{E-cond-E2}, then $\Theta_{X,Y}$ is skew-symmetric: $g(\Theta_{X,Y}Z,W)=-g(\Theta_{X,Y}W,Z)$.



\begin{proposition}\label{P-RP-properties}
For an Einstein's connection, we get
\begin{enumerate}
\item[$(i)$] ${R}^{f}_{X,Y} Z = R^g_{\,{f}X,{f}Y} Z +\Theta_{X,Y}Z$;
\quad $({R}^{f}_{X,Y}\,\omega)(Z) = -\omega({R}^{f}_{X,Y} Z +\Theta_{X,Y}Z)$ for $\omega\in\Lambda^1(M)$;
\vskip1mm
if \eqref{E-cond-E2} holds, then $g({R}^{f}_{X,Y}\,Z,W) = -g({R}^{f}_{X,Y}\,W,Z)$;
\vskip1mm
\item[$(ii)$] ${R}^{f}_{X,Y}\,\psi={\mathfrak{D}}^{f}(X,Y)\psi$ \ for any $\psi\in C^2(M)$; \ \ if \eqref{E-cond-E2} holds, then ${R}^{f}_{X,Y}\,g=0$;
\vskip1mm
\item[$(iii)$] for every $(1,k)$-tensor $S$ we have
\begin{align}\label{E-Rf-S}
\notag
 \hskip-6mm ({R}^{f}_{X,Y}\,S)(Z_1,\ldots,Z_k) =& (R^g_{\,{f}X,{f}Y}\,S)(Z_1,\ldots,Z_k) +\Theta_{X,Y} S(Z_1,\ldots,Z_k) \\
 & -\sum\nolimits_{\,i} S(Z_1,\ldots \Theta_{X,Y} Z_i, \ldots, Z_k).
\end{align}
\item[$(iv)$] ${R}^{f}(X,Y,Z,W) = R^g({f}X,{f}Y,Z,W) +g(\Theta_{X,Y}Z, W)$;
\vskip1mm
\item[$(v)$] ${R}^{f}(X,Y,Z,W) = -{R}^{f}(Y,X,Z,W)$;

if \eqref{E-cond-E2} holds, then
${R}^{f}(X,Y,Z,W) = -{R}^{f}(X,Y,W,Z)$.
\end{enumerate}
\end{proposition}

\begin{proof}
 (i)
 By definition of ${R}^{f}_{X,Y}$,
we~get the first two equalities.
Since $\Theta_{X,Y}:TM\to TM$ is skew-symmetric, ${R}^{f}_{X,Y}$ is also skew-symmetric.

\noindent
(ii) We calculate
\begin{equation*}
 {R}^{f}_{X,Y}\,\psi = {f}X({f}Y(\psi))
 -{f}Y({f}X(\psi)) -({f}[X,Y]_{f})\psi
 = {\mathfrak{D}}^{f}(X,Y)\psi.
\end{equation*}
Next, using (i) we obtain
\begin{eqnarray*}
 && g({R}^{f}_{X,Y}\,Z, W) =
 g(R^g_{\,{f}X,{f}Y}\,Z, W) +g(\Theta_{X,Y}Z,\,W).
\end{eqnarray*}
Similarly, $g({R}^{f}_{X,Y}W, Z)
= g(R^g_{\,{f}X,{f}Y}W, Z)+g(\Theta_{X,Y}W, Z)$.
By this, \eqref{Curv-S} and $\nabla^g_{{\mathfrak{D}}^{f}(X,Y)}\,g =0$, we~get
\begin{align*}
 & ({R}^{f}_{X,Y}\,g)(Z,W) = -g({R}^{f}_{X,Y}\,Z,W) - g(Z, {R}^{f}_{X,Y}\,W) \\
 & = (R^g_{\,{f}X,{f}Y}\,g)(Z,W) -g(\Theta_{X,Y}Z, W) -g(\Theta_{X,Y}W, Z) \\
 &
 = (R^g_{\,{f}X,{f}Y}\,g)(Z,W) .
\end{align*}
Using $R^g_{\,{f}X,{f}Y}\,g=0$ in the above,
we obtain ${R}^{f}_{X,Y}\,g=0$.

\noindent
(iii) From the above and \eqref{Curv-S} the required result \eqref{E-Rf-S} follows.

\noindent
(iv) The equalities follow from \eqref{Curv-P} and (i).

\noindent
(v) Since ${R}^{f}_{X,Y}\,Z=-{R}^{f}_{Y,X}\,Z$, see (i), the first equality follows. For the second one, we use~\eqref{E-cond-E2}:
 $g({R}^{f}_{X,Y} Z,\,Z) =0$;
thus, the claim follows from
$g({R}^{f}_{X,Y}(Z+W),\,Z+W)=0$.
\end{proof}

Similarly, we define the ${f}$-curvature operator of the conjugate ${f}$-connection $\breve\nabla^{f}$,
\begin{equation*}
 \breve R^{f}_{X,Y}\,Z =\breve\nabla_{X}^{f}\,\breve\nabla^{f}_{Y}\,Z-\breve\nabla^{f}_{Y}\,\breve\nabla^{f}_{X}\,Z -\breve\nabla^{f}_{\,[X,Y]^\smile_{f}}Z,\quad X,Y,Z\in\mathfrak{X}_M.
\end{equation*}
If we assume \eqref{E-cond-E2}, i.e., $K^*_X=-K_X$, then using \eqref{E-brave-nabla} we get $R^{f}_{X,Y} = \breve R^{f}_{X,Y}$.

\begin{proposition} For an Einstein's connection
satisfying \eqref{E-cond-E2},
we~get
\begin{equation}\label{E-Ric-K}
 {\rm Ric}^{f}(X,Y) = {\rm Ric}^{0}(X,Y)
 + g\big(\sum\nolimits_{\,i} \Theta_{X,e_i} e_{i},\, Y\big),
\end{equation}
where
${\rm Ric}^{{0}}(X,Y) :=
 \sum\nolimits_{\,i}R^g({f}X, {f}e_{i}, e_{i}, Y)$
is the Ricci tensor of $\nabla^{{0}}$.
\end{proposition}

\begin{proof}
Using (iv) of Proposition~\ref{P-RP-properties},
\eqref{E-tordfnew2} and \eqref{E-cond-E2}, we have
\begin{align*}
 {\rm Ric}^{f}(X,Y) &= \sum\nolimits_{\,i}{R}^{f}(X, e_{i}, e_{i}, Y)
 ={\rm Ric}^{0}(X,Y) + \sum\nolimits_{\,i} g( \Theta_{X,e_i} e_{i},\,Y) .
\end{align*}
From the above the claim follows.
\end{proof}


The skew-symmetric tensor ${f}$ induces a self-adjoint endomorphism ${\cal P}$ of $\Lambda^{2}(M)$:
\[
 {\cal P}(X\wedge Y) ={f}X\wedge {f}Y,\quad
 {\cal P}^*(X\wedge Y)={f}^*X\wedge {f}^*Y
 \quad \Longrightarrow \quad {\cal P}^*={\cal P}.
\]
The curvature operator $R^g_{X,Y}$ can be seen as a self-adjoint linear operator ${\cal R}^g$
on the space $\Lambda^{2}(M)$ of bivectors,
e.g., \cite{besse,Peter}.
 Similarly, we consider $R^{f}_{X,Y}$ as a linear operator,
or as a corresponding bilinear form on $\Lambda^{2}(M)$.
For this,
using assumption \eqref{E-cond-E2},
define a linear operator ${\cal K}$ on $\Lambda^{2}(M)$~by
\[
 g({\cal K}(X\wedge Y),Z\wedge W) := g(\Theta_{X,Y}Z,\ W).
\]
Put ${\cal R}^{f} := {\cal R}^{g}\,{\cal P} +{\cal K}$ and
$\breve{\cal R}^{f} := {\cal R}^{g}\,{\cal P} -{\cal K}$, i.e.,
\begin{align*}
 & {\cal R}^{f}(X\wedge Y)={\cal R}^{g}\,{\cal P}(X\wedge Y)+{\cal K}(X\wedge Y)
 ={\cal R}^{0}({f}X\wedge {f}Y)+{\cal K}(X\wedge Y),\\
 & {\cal R}^{f}(X\wedge Y,Z\wedge W) = g({\cal R}^{\,{f}}(X\wedge Y),Z\wedge W),\\
 &\breve{\cal R}^{f}(X\wedge Y)
 ={\cal R}^{g}\,{\cal P}(X\wedge Y)-{\cal K}(X\wedge Y)
 ={\cal R}^{g}({f}X\wedge {f}Y)-{\cal K}(X\wedge Y) ,\\
 & \breve{\cal R}^{f}(X\wedge Y,Z\wedge W) =
 g(\breve{\cal R}^{\,{f}}(X\wedge Y),Z\wedge W).
\end{align*}
The above ${\cal R}^{f}$ generalizes
${\cal R}^{0} := {\cal R}^g\,{\cal P}$ associated with $\nabla^0$, see \cite{rp-2}.
Using known properties of ${\cal R}^{g}$ and property (iv) of $R^{f}$, see Proposition~\ref{P-RP-properties}, we have
\begin{align*}
 & g({\cal R}^{f}(X\wedge Y),Z\wedge W) =
 g({\cal R}^{g}({f}X\wedge {f}Y)+{\cal K}(X\wedge Y),\,Z\wedge W) = R^{f}(X,Y,W,Z),\\
 & g(\breve{\cal R}^{f}(X\wedge Y),Z\wedge W) =
 g({\cal R}^{g}({f}X\wedge {f}Y)-{\cal K}(X\wedge Y),\,Z\wedge W) = \breve R^{f}(X,Y,W,Z).
\end{align*}

\subsection{The Weitzenb\"{o}ck $f$-curvature operator}
\label{sec:Ric}

\noindent
Here, we use  an Einstein's connection $\nabla$ to introduce
the central concept of the Bochner technique: the Weitzenb\"{o}ck $f$-curvature operator on tensors,
gene\-ralizing the
operator \eqref{E-Ric},
for the case of a manifold $(M,G=g+F)$ with an $f$-connection~$\nabla^{f}$.

\begin{definition}\rm
 Define the \emph{Weitzenb\"{o}ck $f$-curvature operator}
 on $(0,k)$-tensors $S$ on a manifold $(M,G=g+F)$~by
\begin{equation}\label{E-Ric-P}
 \Re^{f}(S)(X_{1},\ldots,X_{k})
 =\sum\nolimits_{\,a=1}^{k}\sum\nolimits_{\,i}({R}^{f}_{\,e_{i},X_{a}}\,S)(\underbrace{X_{1},\ldots,e_{i}}_{a},\ldots,X_{k}).
\end{equation}
The operators $\breve{\Re}^{\,{f}}$ and ${\Re}^{0}$ are defined similarly using
$\breve\nabla^{f}$ and $\nabla^{0}$.
\end{definition}

For a form $\omega\in\Lambda^k(M)$, the $\Re^{f}(\omega)$ is skew-symmetric, see \eqref{Curv-S} with ${\mathfrak{D}}^{f}=0$.
For $k=1$, $\Re^{f}$ reduces to ${\rm Ric}^{f}$.
For $k\ge2$, using \eqref{Curv-S}, the formula from \eqref{E-Ric-P} reads as
\begin{align}\label{E-Ric-Pb}
 \nonumber
 \Re^{f}(S)(X_{1},\ldots,X_{k})
 =& -2\sum\limits_{\,i,j,a;\,b<a}R^{f}(e_i,X_a,e_j,X_b)\cdot S(\underbrace{X_1,\ldots,e_j}_{b},\underbrace{\ldots, e_i}_{a-b},\ldots,X_k)\\
 & +\,
 \sum\nolimits_{\,a=1}^{k}\sum\nolimits_{\,i}
 {\rm Ric}^{f}(e_i,X_a)\cdot S(\underbrace{X_1,\ldots, e_i}_{a},\ldots, X_k).
\end{align}


The following lemma represents $\Re^{f}$ using $\Re^{\,{0}}$ and the difference tensor $K$.

\begin{lemma}
Let an Einstein's connection $\nabla$ on a manifold $(M,G=g+F)$ satisfy \eqref{E-cond-E2}.
Then
\begin{equation}\label{E-Ric-hat-Ric}
 \Re^{f} \omega=\Re^{\,{0}}\omega - \mathfrak{K}\,\omega,
\end{equation}
where the operator $\mathfrak{K}$ acts on forms $\omega\in\Lambda^k(M)$ for $k\ge2$ by
\begin{align}\label{E-K-frak}
\notag
 (\mathfrak{K}\,\omega)(X_{1},\ldots, X_{k}) =& \ 2\sum\nolimits_{\,i,j,\,b<a} g(\Theta_{e_i, X_a}e_{j}, \ X_b)\,
 \omega(\underset{b}{\underbrace{X_{1},\ldots,e_{j}}},\underset{a-b}{\underbrace{\ldots,e_{i}}}\ldots,X_{k})  \\
 &  +\sum\nolimits_{\,a=1}^{k}\sum\nolimits_{\,i} g\big(\sum\nolimits_{\,j} \Theta_{e_i,\,e_{j}}e_{j},\,X_a\big)
 \,\omega(\underbrace{X_{1},\ldots, e_i}_{a},\ldots,X_{k}) ,
\end{align}
and $(\mathfrak{K}\,\omega)(X) = \sum\nolimits_{\,i}
g\big(\sum\nolimits_{\,j} \Theta_{e_i,\,e_{j}}e_{j},\,X\big)\,\omega(e_i)$ for $k=1$.
\end{lemma}

\begin{proof}
Using (iv) of Proposition~\ref{P-RP-properties} and \eqref{E-Ric-K}, we find
\begin{align}\label{E-K-frak-tmp}
\notag
 & {R}^{f}(e_i,X_a,e_j,X_b) = {R}^{0 }(e_i,X_a,e_j,X_b) +g(\Theta_{e_i, X_a}e_{j}, X_b),\\
 &\hskip-3mm {\rm Ric}^{f}(e_i,X_a) = {\rm Ric}^{\,{0}}(e_i,X_a)
 + g(\sum\nolimits_{\,j} \Theta_{e_i,\,e_{j}}e_{j},\,X_a) .
\end{align}
Using \eqref{E-K-frak-tmp} in \eqref{E-Ric-Pb}, by linearity in the curvature and \eqref{E-Rf-S} we get \eqref{E-Ric-hat-Ric} and \eqref{E-K-frak}.
\end{proof}

 We generalizes \eqref{E-Wei0} and \cite[Theorem 2]{rp-2} to the case of an Einstein's connection.

\begin{theorem}
Let an Einstein's connection $\nabla$ satisfy \eqref{E-cond-E2}.
Then the
\textit{Weitzenb\"{o}ck type decomposition formula} is~valid for any~$\omega\in\Lambda^k(M)$:
\begin{equation}\label{E-Wei}
 \Delta^{f}_H\,\omega = \breve\nabla^{*{f}}\,\nabla^{f}\,\omega +\Re^{f}(\omega).
\end{equation}
\end{theorem}

\begin{proof}
Using the conditions, $E=E^*=0$, Lemma~\ref{L-03} and \eqref{E-58-59}, we find
\begin{align*}
 d^{f}\,\breve{\delta}^{f}\omega( X_{1},\ldots,X_{k}) & =-\sum\nolimits_{\,j} d^{f}\,\nabla_{e_{j}}^{\,{f}}\omega( e_{j},X_{1},\ldots,X_{k}) \\
 &=\sum\nolimits_{\,j}\sum\nolimits_{a=0}^{k-1}(-1)^{a}\nabla_{X_{a+1}}^{\,{f}}\nabla_{e_{j}}^{\,{f}}\omega(e_{j},X_{1},\ldots X_{a},X_{a+2}\ldots,X_{k}) \\
 &=-\sum\nolimits_{\,j}\sum\nolimits_{a=0}^{k-1}\nabla_{X_{a+1}}^{\,{f}}
 \nabla_{e_{j}}^{f}\omega\big(\underset{a+1}{\underbrace{X_{1},\ldots e_{j}}},\ldots,X_{k}\big) \\
 &= -\sum\nolimits_{\,j,a}((\nabla^{f})^2_{X_{a+1},e_{j}}\,\omega) \big( \underset{a+1}{\underbrace{X_{1},\ldots e_{j}}},\ldots,X_{k}\big), \\
 \breve{\delta}^{f}\,d^{f}\omega( X_{1},\ldots,X_{k}) & = \nabla^{*{f}}(d^{f}\omega)( X_{1},\ldots,X_{k}) \\
& = -\sum\nolimits_{\,j}\nabla_{e_{j}}^{f}( d^{f}\omega)( e_{j},X_{1},\ldots,X_{k})
  = -\sum\nolimits_{\,j}\nabla^{f}_{e_{j}}\,\nabla^{f}_{e_{j}}\omega( X_{1},\ldots,X_{k}) \\
& + \sum\nolimits_{\,j}\sum\nolimits_{a=0}^{k-1}(-1)^{a}\nabla^{f}_{e_{j}}\,\nabla^{f}_{X_{a+1}}\omega( e_{j},X_{1},\ldots,X_{a},X_{a+2},\ldots,X_{k}) \\
& =(\nabla^{*{f}}\,\nabla^{f}\omega)( X_{1},\ldots,X_{k})
+\sum\nolimits_{\,j,a}( (\nabla^{f})^2_{e_{j},X_{a+1}}\,\omega)(\underset{a+1}{\underbrace{X_{1},\ldots e_{j}}},\ldots,X_{k}) .
\end{align*}
From the above we have
\begin{align*}
\Delta^{f}_H\,\omega = (\nabla^{*{f}}\nabla^{f}\omega)( X_{1},\ldots,X_{k})
{+}\sum\limits_{\,j,a}\big( (\nabla^{f})^2_{e_{j},X_{a+1}}\,\omega
-(\nabla^{f})^2_{X_{a+1},e_{j}}\,\omega\big)(\underset{a+1}{\underbrace{X_{1},\ldots e_{j}}},\ldots,X_{k}).
\end{align*}
Thus, we have \eqref{E-Wei}.
\end{proof}

\begin{example}\rm
For vector fields and 1-forms, $\Re^{f}$ reduces to the kind of usual Ricci curvature, see \eqref{E-Ric0} and \eqref{E-Ric-Pb}.
We have $\Re^{f}(\omega)(X)=\omega({\rm Ric}^{f}(X))$ for any $\omega\in\Lambda^1(M)$; thus, \eqref{E-Wei} reads as
\[
  \Delta^{f}_H\,\omega = \breve\nabla^{*{f}}\,\nabla^{f}\omega +{\rm Ric}^{f}(\omega^\sharp).
\]
\end{example}

 For a bivector $X\wedge Y\in\Lambda^{2}(M)$, consider a map ${\cal R}^{f}(X\wedge Y) : \mathfrak{X}_M\to \mathfrak{X}_M$, given by
\begin{eqnarray*}
 g({\cal R}^{f}(X\wedge Y) Z,W) \eq g({\cal R}^{f}(X\wedge Y), W\wedge Z) = R^{f}(X,Y,Z,W) \\
 \eq R^{g}({f}X,{f}Y,Z,W) +g(\Theta_{X,Y}Z, W) ,
\end{eqnarray*}
see (iv) of Proposition~\ref{P-RP-properties}.
Since bivectors are generators of the vector space $\Lambda^{2}(M)$, we get
a map ${\cal R}^{f}(\xi): \mathfrak{X}_M\to \mathfrak{X}_M$, where $\xi\in\Lambda^{2}(M)$ (similarly to algebraic curvature operator ${\cal R}^{g}(\xi)$).

\begin{lemma}\label{lmab+}
For an Einstein's connection $\nabla$ satisfying \eqref{E-cond-E2},
the map ${\cal R}^{f}(\xi)$, where $\xi\in\Lambda^{2}(M)$, is skew-symmetric:
 $g({\cal R}^{f}(\xi)W,\,Z) =-g({\cal R}^{f}(\xi) Z,\,W)$.
\end{lemma}

\begin{proof}
It suffices to check the statement for the generators.
Note that by condition \eqref{E-cond-E2}, $K_X$ is skew-symmetric.
Hence, using Proposition~\ref{P-RP-properties} and the conditions,
we get
\begin{align*}
 & g({\cal R}^{f}(X\wedge Y) Z,\ W) = R^{g}({f}X,{f}Y,Z,W) +g(\Theta_{X,Y}Z,\ W) \\
 & = -R^{g}({f}X,{f}Y,W,Z) -g(\Theta_{X,Y}W,\ Z) = -g({\cal R}^{f}(X\wedge Y) W,\ Z).
\end{align*}
Thus, the statement follows.
\end{proof}

The associated ${f}$-\textit{curvature operator} is given by
\[
 g({\cal R}^{f}(X\wedge Y),\, Z\wedge W) = R^{g}({f}X,{f}Y,W,Z) -g(\Theta_{X, Y}Z, W).
\]
To simplify calculations, we assume that $\mathfrak{so}(TM)$ is endowed with metric induced from $\Lambda^{2}(M)$. If~$L\in\mathfrak{so}(TM)$, then
\begin{equation}\label{E-L-S}
 (L\,S)(X_1,\ldots, X_k) = -\sum\nolimits_{\,i} S(X_1,\ldots,L(X_i),\ldots, X_k).
\end{equation}
Let $\{\xi_{a}\}$ be an orthonormal base of skew-symmetric transformations such that $(\xi_{a})_{x}\in\mathfrak{so}(T_{x}M)$
for $x$ in an open set $U\subset M$.
By \eqref{E-L-S}, for any $(0,k)$-tensor~$S$,
\[
 (\xi_\alpha S)(X_1,\ldots, X_k) = -\sum\nolimits_{\,i} S(X_1,\ldots,\xi_\alpha(X_i),\ldots, X_k);
\]
The ${\cal R}^{f}(X\wedge Y)$ on $\Lambda^{2}(M)$ can be decomposed using $\{\xi_a\}$.

\begin{lemma}\label{L-0-3P1+}
Let $\nabla$ be an Einstein's connection satisfying \eqref{E-cond-E2}.
Then
\begin{eqnarray*}
 {\cal R}^{f}(X\wedge Y) \eq -\sum\nolimits_{\,\alpha } \big(g({\cal P}\,{\cal R}^{g}(\xi_{\alpha})X,Y)
 +g({\cal K}(X\wedge Y),\xi_{\alpha})\big)\xi_{\alpha} \\
 \eq -\sum\nolimits_{\,\alpha}\big(g({\cal R}^{g}(\xi_{\alpha}){f}X,{f}Y)+g({\cal K}(X\wedge Y),\xi_{\alpha})\big)\xi_{\alpha}.
\end{eqnarray*}
\end{lemma}

\proof Using $({\cal R}^{f})^*={\cal P}{\cal R}^{g}$ and Lemma~\ref{lmab+}, we have:
\begin{eqnarray*}
 {\cal R}^{f}(X\wedge Y) \eq \sum\nolimits_{\,\alpha }
 g({\cal R}^{f}(X\wedge Y),\xi_{\alpha}) \,\xi_{\alpha}\\
 \eq\sum\nolimits_{\,\alpha}\big(g({\cal P}{\cal R}^{g}(\xi_{\alpha}),X\wedge Y)
 +g({\cal K}(X\wedge Y),\xi_{\alpha})\big)\xi_{\alpha} \\
 \eq -\sum\nolimits_{\,\alpha}\big(g({\cal R}^{g}(\xi_{\alpha}){f}X,{f}Y)
 +g({\cal K}(X\wedge Y),\xi_{\alpha})\big)\xi_{\alpha},
 \hskip 30mm\Box
\end{eqnarray*}

\noindent\
Lemma~\ref{L-0-3P1+} allows us to rewrite the Weitzenb\"{o}ck
$f$-curvature operator \eqref{E-Ric-P}.

\begin{proposition}
Let $\nabla$ be an Einstein's connection satisfying \eqref{E-cond-E2},
and $S$ a $(0,k)$-tensor on $(M,G=g+F)$. Then
\begin{align}\label{Eq-prop10}
 \Re^{f}(S)=-\sum\nolimits_{\,\alpha}{\cal R}^{f}(\xi_{a})(\xi_{a}S),\qquad
 (\Re^{f}(S))^*=\Re^{-{f}}(S).
\end{align}
\end{proposition}

\begin{proof}
We prove \eqref{Eq-prop10}$_1$ using similar arguments as in the proof of \cite[Lemma~9.3.3]{Peter}:
\begin{align*}
 &\quad \Re^{f}(S)(X_{1},\ldots ,X_{k}) =\sum\nolimits_{\,i,j}({\cal R}^{f}(e_{j}\wedge X_{i})S)
   (\underbrace{X_{1},\ldots ,e_{j}}_i,\ldots ,X_{k}) \\
 &  =-\sum\nolimits_{\,i,j,\alpha}
   \big(g({\cal P}\,{\cal R}^{g}(\xi_{\alpha})e_{j},X_{i})+g({\cal K}(e_{j}\wedge X_{i}),\xi_{\alpha})\big)
   (\xi_{\alpha}S)(X_{1},\ldots ,e_{j},\ldots ,X_{k}) \\
 & =-\sum\nolimits_{\,i,j,\alpha}(\xi_{\alpha}S)(X_{1},\ldots,
 \big(g({\cal P}{\cal R}^{g}(\xi_{\alpha})e_{j},X_{i}) e_{j}+g({\cal K}(e_{j}\wedge X_{i}),\xi_{\alpha})\big),\ldots ,X_{k}) \\
 & =-\sum\nolimits_{\,i,j,\alpha}(\xi_{\alpha}S)(X_{1},\ldots,g(e_{j},{\cal R}^{f}(\xi_{\alpha})X_{i}) e_{j},\ldots ,X_{k}) \\
 & =-\sum\nolimits_{\,i,\alpha}(\xi_{\alpha}S)(X_{1},\ldots,{\cal R}^{f}(\xi_{\alpha})X_{i},\ldots ,X_{k})
  =-\sum\nolimits_{\,\alpha}({\cal R}^{f}(\xi_{\alpha})(\xi_{\alpha}S))(X_{1},\ldots ,X_{k}).
\end{align*}%
Since ${\cal R}^{g}:\Lambda^{2}(M)\to \Lambda^{2}(M)$ is self-adjoint, there is a local orthonormal base $\{\xi_{a}\}$ of $\Lambda^{2}(M)$
such that ${\cal R}^{g}(\xi_{a})=\lambda_{a}\,\xi_{a}$. For this base and any $(0,k)$-tensors $S_1$ and $S_2$, we~get
\begin{eqnarray}\label{E-RicP-SS}
\nonumber
 g(\Re^{f}(S_2),\, S_1) \eq-\sum\nolimits_{\,\alpha}
 g({\cal R}^{f}(\xi_{\alpha})(\xi_{\alpha}S_2),S_1)
   =-\sum\nolimits_{\,\alpha }g(\xi_{\alpha}S_2,\,({\cal R}^{f})^*(\xi_{\alpha})S_1) \\
\nonumber
 \eq \sum\nolimits_{\,\alpha}g(\xi_{\alpha}S_2,\,({\cal P}{\cal R}^{g} +{\cal K})(\xi_{\alpha})(S_1)) \\
 \eq \sum\nolimits_{\,\alpha}\lambda_{\alpha}g({\cal P}(\xi_{\alpha}S_2),\,\xi_{\alpha}S_1)
 + \sum\nolimits_{\,\alpha}g({\cal K}^*(\xi_{\alpha}S_2),\,\xi_{\alpha}S_1),
\end{eqnarray}
and, similarly,
\begin{eqnarray*}
 g(S_2,\,\Re^{-{f}}(S_1)) = \sum\nolimits_{\,\alpha}\lambda_{\alpha}
 g(\,\xi_{\alpha}S_2,\,{\cal P}(\xi_{\alpha}S_1))
 + \sum\nolimits_{\,\alpha} g(\xi_{\alpha}S_2,\,{\cal K}(\xi_{\alpha}S_1)) \\
 =\sum\nolimits_{\,\alpha}\lambda_{\alpha}g({\cal P}(\xi_{\alpha}S_2),\,\xi_{\alpha}S_1)
 + \sum\nolimits_{\,\alpha}g({\cal K}^*(\xi_{\alpha}S_2),\,\xi_{\alpha}S_1).
\end{eqnarray*}
Thus, \eqref{Eq-prop10}$_2$ follows.
\end{proof}

The following result generalizes
\cite[Corollary 9.3.4]{Peter}.

\begin{proposition}\label{P-RP-ge0}
Let $\nabla$ be an Einstein's connection on a manifold $(M,G=g+F)$ satisfying \eqref{E-cond-E2}.

a)~If $g({\cal R}^{f}(S),S)\ge 0$ for any $(0,k)$-tensor $S$, then $g(\Re^{f}(S),S)\ge 0$.

b)~If $g({\cal R}^{f}(S),S)\ge -\varepsilon\,\|S\|^{2}$ for any $(0,k)$-tensor $S$, where $\varepsilon>0$, then

\quad $g({\Re}^{\,{f}}(S),S)\ge -\varepsilon\,C\,\|S\|^{2}$,
where a constant $C$ depends only on the type of~$S$.
\end{proposition}

\begin{proof}
By \eqref{E-RicP-SS}, for a local orthonormal base $\{\xi_{\alpha}\}$ of $\Lambda^{2}(M)$ such that ${\cal R}^g(\xi_{\alpha})=\lambda_{\alpha}\xi_{\alpha}$, we~get
\begin{eqnarray*}
 g(\Re^{f}(S),S)
 \eq \sum\nolimits_{\,\alpha}\lambda_{\alpha}g({\cal P}(\xi_{\alpha}S),\,\xi_{\alpha}S)
 +\sum\nolimits_{\,\alpha}g({\cal K}(\xi_{\alpha}S),\,\xi_{\alpha}S)\\
 \eq \sum\nolimits_{\,\alpha}g({\cal P}(\xi_{\alpha}S),\,{\cal R}^g(\xi_{\alpha}S))
 +\sum\nolimits_{\,\alpha}g({\cal K}(\xi_{\alpha}S),\,\xi_{\alpha}S) \\
 \eq \sum\nolimits_{\,\alpha}g({\cal R}^{f}(\xi_{\alpha}S),\,\xi_{\alpha}S).
\end{eqnarray*}
By conditions, $g({\cal R}^{f}(\xi_{\alpha}S),\,\xi_{\alpha}S)\ge 0$ for all $\alpha$,
thus, $g(\Re^{f}(S),S)\ge 0$, and the first claim~follows.
There is a constant $C>0$ depending only on the type of the tensor and $\dim M$ such that
 $C\|S\|^2\ge \sum\nolimits_{\,\alpha}\|\xi_{\alpha}S\|^2$,
see~\cite[Corollary~9.3.4]{Peter}.
By conditions, $g({\cal R}^{f}(\xi_{\alpha}S),\,\xi_{\alpha}S)\ge -\varepsilon\,\|\xi_{\alpha}S\|^2$ for all~$\alpha$.
The~above
 yields
$g({\cal R}^{f}(\xi_{\alpha}S),\,\xi_{\alpha}S)\ge
-\varepsilon\,C\,\|S\|^{2}$ -- thus, the second~claim.
\end{proof}


\begin{theorem}
Let conditions \eqref{E-cond-E2}
and \eqref{E-cond-PP-stat} hold for an Einstein's connection $\nabla$ on a closed manifold $(M,G=g+F)$.
If $g({\cal R}^{f}(\omega),\omega)\ge 0$ for an
${f}$-harmonic form $\omega\in\Lambda^k(M)$, then~$\nabla^{f}\omega=0$.
\end{theorem}

\begin{proof} By conditions and Proposition~\ref{P-RP-ge0}(a), $g(\Re^{f}(\omega),\omega)\ge 0$ is true.
By \eqref{E-Wei}, since $\Delta^{f}_H\,\omega=0$, we get $g(\breve\nabla^{*{f}}\,\nabla^{f}\omega,\omega)\le0$.
By Proposition~\ref{P-max}, where \eqref{E-cond-PP-stat} is used, we have $\nabla^{f}\omega=0$.
\end{proof}

\section*{\large Conclusion}

\noindent
The main contribution of this paper (four theorems and ten propositions)
is the development of Bochner's technique for a manifold $(M,G=g+F)$ with an Einstein's connection.
To generalize some geometrical analysis tools for an Einstein's connection, we assume some additional conditions, for example, \eqref{E-cond-E2} and \eqref{E-cond-PP-stat}, for $g$, $F$ and the difference tensor $K$,
give an example of the Einstein's connection.
We~define the $f$-curvature operator and the Weitzenb\"{o}ck $f$-curvature operator for NGT spaces, and prove vanishing theorems on the null space of the Bochner and Hodge $f$-Laplacians. We delegate the following for further study:
considering NGT spaces associated with special classes of weak metric $f$-manifolds, see \cite{rov-survey24,rst-139,rst-63}, and find more applications in NGT and mathematical physics.

\bigskip

\noindent {\bf Acknowledgments}.
The author expresses his sincere gratitude to
Professor Milan Zlatanovi\'c (University of Ni\v s) for numerous fruitful discussions of NGT and its applications. His comments significantly improved both the accuracy and clarity of the manuscript.

\end{document}